\documentclass{amsart}

 \usepackage[nobysame]{amsrefs}

 \AtBeginDocument{\def\MR#1{}}

\usepackage[mathlines]{lineno}

\usepackage[english]{babel} 
\usepackage[UKenglish]{datetime}

\usepackage{ae, aecompl}
\usepackage{amscd}
\usepackage{amsfonts}
\usepackage{amsmath}
\usepackage{amssymb}
\usepackage{amsthm}
\usepackage{amsxtra}
\usepackage{cases}
\usepackage{chngcntr}
\usepackage{cite}
\usepackage{color}
\usepackage{graphicx}
\usepackage{latexsym}
\usepackage{bbm} 
\usepackage{bm}
\usepackage{bbold}
\usepackage{mathtools}
\usepackage{microtype}
\usepackage{qsymbols}
\usepackage[Symbolsmallscale]{upgreek} 
\usepackage[active]{srcltx}
\usepackage{tikz}
\usepackage{url}
\usepackage{verbatim} 

\usepackage{enumitem}
\usepackage{cleveref}

\newcommand{\bbfont}{\mathbbm}

\usepackage{ifthen}

\makeatletter
\newcommand{\tfs}[1]
{
\ifthenelse{\equal{\f@shape}{n}}{\ensuremath{\mathrm{#1}}}
	{\ifthenelse{\equal{\f@shape}{sc}}{\ensuremath{\mathrm{#1}}}
		{\ifthenelse{\equal{\f@shape}{it}}{\ensuremath{\mathit{#1}}}
			{\ifthenelse{\equal{\f@shape}{sl}}{\ensuremath{\mathit{#1}}}{}	
			}
		}
	}
}
\makeatother

\renewcommand{\tfs}{\mathrm}

\makeatletter
\newcommand{\btfs}[1]
{
\ifthenelse{\equal{\f@shape}{n}}{\ensuremath{\mathrm{#1}}}
	{\ifthenelse{\equal{\f@shape}{sc}}{\ensuremath{\mathrm{#1}}}
		{\ifthenelse{\equal{\f@shape}{it}}{\ensuremath{\mathit{#1}}}
			{\ifthenelse{\equal{\f@shape}{sl}}{\ensuremath{\mathit{#1}}}{}	
			}
		}
	}
}
\makeatother

\newcommand{\CC}{{\bbfont C}}

\newcommand{\FF}{{\bbfont F}}

\newcommand{\NN}{{\bbfont N}}

\newcommand{\RR}{{\bbfont R}}

\newcommand{\ZZ}{{\bbfont Z}}

\newcommand{\upd}{{\mathrm{d}}}

\newcommand{\ulp}{{\textup{(}}}
\newcommand{\urp}{{\textup{)}}}

\newcommand{\uppars}[1]{\ulp #1\urp}

\newcommand{\abs}[1]{{\lvert #1 \rvert}}
\newcommand{\norm}[1]{{\lVert #1 \rVert}}

\newcommand{\braces}[1]{{\{ #1\}}}

\newcommand{\lrabs}[1]{{\left\lvert #1 \right\rvert}}

\newcommand{\lrangles}[1]{{\left\langle #1\right\rangle}}
\newcommand{\lrbraces}[1]{{\left\{ #1\right\}}}

\newcommand{\set}[1]{\braces{#1}}

\newcommand{\desset}[1]{\braces{\,#1\,}}
\newcommand{\lrdesset}[1]{\lrbraces{\,#1\,}}

\newcommand{\di}[1]{\,\upd #1}

\newcommand{\cont}{\tfs{C}}
\newcommand{\conto}{\cont_0}
\newcommand{\contb}{\cont_{\tfs{b}}}
\newcommand{\contc}{\cont_{\tfs{c}}}

\newcommand{\bounded}{\tfs{B}}
\newcommand{\linear}{\tfs{L}}
\newcommand{\regular}{\linear_{\tfs{r}}}

\newcommand{\Ell}{\tfs{L}}

\renewcommand{\Re}{\operatorname{Re}} 

\theoremstyle{plain}

\newtheorem{theorem}{Theorem}[section]
\newtheorem{proposition}[theorem]{Proposition}
\newtheorem{lemma}[theorem]{Lemma}
\newtheorem{corollary}[theorem]{Corollary}

\newtheorem*{theorem*}{Theorem}
\newtheorem*{proposition*}{Proposition}
\newtheorem*{lemma*}{Lemma}
\newtheorem*{corollary*}{Corollary}
\newtheorem*{conjecture*}{Conjecture}
\newtheorem*{assumption*}{Assumption}
\newtheorem*{hypothesis*}{Hypothesis}
\newtheorem*{question*}{Question}
\newtheorem*{problem*}{Problem}
\newtheorem*{task*}{Task}
\newtheorem*{addendum*}{Addendum}
\newtheorem*{idea*}{Idea}
\newtheorem*{suggestion*}{Suggestion}
\newtheorem*{context*}{Context}
\newtheorem*{exercise*}{Exercise}

\theoremstyle{definition}

\newtheorem{definition}[theorem]{Definition}
\newtheorem{example}[theorem]{Example}
\newtheorem{remark}[theorem]{Remark}

\newtheorem*{definition*}{Definition}
\newtheorem*{example*}{Example}
\newtheorem*{remark*}{Remark}

\setlist[enumerate,1]{label=\textup{(\arabic*)},ref=\arabic*}
\setlist[enumerate,2]{label=\textup{(\alph*)},ref=\arabic{enumi}.\alph*}
\setlist[enumerate,3]{label=\textup{(\roman*)},ref=\arabic{enumi}.\alph{enumii}.\roman*}
\setlist[enumerate,4]{label=\textup{(\Alph*)},ref=\arabic{enumi}.\alph{enumii}.\roman{enumiii}.\Alph*}

\crefname{theorem}{Theorem}{Theorems}
\crefname{proposition}{Proposition}{Propositions}
\crefname{lemma}{Lemma}{Lemmas}
\crefname{corollary}{Corollary}{Corollaries}
\crefname{conjecture}{Conjecture}{Conjectures}
\crefname{definition}{Definition}{Definitions}
\crefname{example}{Example}{Examples}
\crefname{remark}{Remark}{Remarks}
\crefname{assumption}{Assumption}{Assumptions}
\crefname{hypothesis}{Hypothesis}{Hypotheses}
\crefname{question}{Question}{Questions}
\crefname{problem}{Problem}{Problems}
\crefname{task}{Task}{Tasks}
\crefname{addendum}{Addendum}{Addenda}
\crefname{idea}{Idea}{Ideas}
\crefname{suggestion}{Suggestion}{Suggestions}
\crefname{context}{Context}{Contexts}
\crefname{exercise}{Exercise}{Exercises}

\crefname{equation}{equation}{equations}
\crefname{enumi}{part}{parts}
\crefname{enumii}{part}{parts}
\crefname{enumiii}{part}{parts}
\crefname{enumiv}{part}{parts}

\numberwithin{equation}{section}

\allowdisplaybreaks 

\usepackage{wasysym}

\newcommand{\circlesign}[1]{
	\mathbin{
		\mathchoice
		{\buildcirclesign{\displaystyle}{#1}}
		{\buildcirclesign{\textstyle}{#1}}
		{\buildcirclesign{\scriptstyle}{#1}}
		{\buildcirclesign{\scriptscriptstyle}{#1}}
	}
}

\newcommand\buildcirclesign[2]{%
	\begin{tikzpicture}[baseline=(X.base), inner sep=0, outer sep=0]
	\node[draw,circle] (X)  {\ensuremath{#1 #2}};
	\end{tikzpicture}%
}

\newcommand{\bil}{\circlesign{\star}}

\newcommand{\il}{\tfs{i}}
\newcommand{\el}{\tfs{e}}

\newcommand{\vl}{E}
\newcommand{\vltwo}{F}

\newcommand{\oadj}[1]{#1^{\sim}}

\newcommand{\res}{R_\tstwo}
\newcommand{\resdual}{\oadj{\res}}

\newcommand{\alg}{A}
\newcommand{\algtwo}{B}

\newcommand{\comp}[1]{#1_{\CC}}

\newcommand{\compalg}{\comp{\alg}}
\newcommand{\compalgtwo}{\comp{\algtwo}}

\newcommand{\compvl}{\comp{\vl}}
\newcommand{\compvltwo}{\comp{\vltwo}}
\newcommand{\compabs}[1]{\comp{\abs{#1}}}
\newcommand{\compnorm}[1]{\comp{\norm{#1}}}

\newcommand{\tstwo}{Y}

\newcommand{\Elloneloc}{\Ell^{1,\tfs{loc}}}

\newcommand{\regularnorm}[1]{\norm{#1}_{\tfs{r}}}

\newcommand{\pos}[1]{#1^+}
\newcommand{\nega}[1]{#1^-}

\newcommand{\supp}[1]{\tfs{supp}\,#1}

\newcommand\group{G}
\newcommand{\ts}{X}
\newcommand{\borelsalg}{{\mathfrak{B}}}
\newcommand{\meas}{{\tfs{M}}}

\newcommand{\EllonegroupR}{\Ell^1(\group,\RR)}
\newcommand{\EllpgroupR}{\Ell^p(\group,\RR)}
\newcommand{\measgroupR}{\tfs{M}(\group,\RR)}

\newcommand{\EllonegroupC}{\Ell^1(\group,\CC)}
\newcommand{\EllpgroupC}{\Ell^p(\group,\CC)}
\newcommand{\measgroupC}{\tfs{M}(\group,\CC)}

\newcommand{\EllonegroupF}{\Ell^1(\group,\FF)}
\newcommand{\EllpgroupF}{\Ell^p(\group,\FF)}
\newcommand{\measgroupF}{\tfs{M}(\group,\FF)}

\newcommand{\conttsR}{\cont(\ts,\RR)}
\newcommand{\contbtsR}{\contb(\ts,\RR)}
\newcommand{\contotsR}{\conto(\ts,\RR)}
\newcommand{\contctsR}{\contc(\ts,\RR)}

\newcommand{\conttsC}{\cont(\ts,\CC)}
\newcommand{\contbtsC}{\contb(\ts,\CC)}
\newcommand{\contotsC}{\conto(\ts,\CC)}
\newcommand{\contctsC}{\contc(\ts,\CC)}

\newcommand{\contogroupR}{\conto(\group,\RR)}
\newcommand{\contcgroupR}{\contc(\group,\RR)}

\newcommand{\contcgroupC}{\contc(\group,\CC)}

\newcommand{\contctsRod}{\contctsR^{\sim}}
\newcommand{\contctstwoR}{\contc(\tstwo,\RR)}
\newcommand{\contctstwoRod}{\contctstwoR^{\sim}}
\newcommand{\contcgroupRod}{\contcgroupR^\sim}

\newcommand{\regularextposmeasts}{\meas_{{\tfs{r}}}(\ts,\borelsalg,\overline{\pos{\RR}})}
\newcommand{\regularposmeasts}{\meas_{{\tfs{r}}}(\ts,\borelsalg,\pos{\RR})}
\newcommand{\regularrealmeasts}{\meas_{{\tfs{r}}}(\ts,\borelsalg,\RR)}

\newcommand{\od}[1]{{#1}^\sim}
\newcommand{\vlod}{\od{\vl}}
\newcommand{\nd}[1]{{#1}^\prime}
\newcommand{\vlnd}{\nd{\vl}}

\newcommand{\ode}{\varphi} 
\newcommand{\Ode}{\Phi}
\newcommand{\rep}{\pi}
\newcommand{\ind}{\chi}

\newcommand{\sgn}{\tfs{sgn}}

\newcommand{\orderboundedops}[1]{\tfs{L}_{\tfs{b}}(#1)}
\newcommand{\regularops}[1]{\tfs{L}_{\tfs{r}}(#1)}

\newcommand{\pairing}[1]{\lrangles{#1}}

\newcommand{\sigmats}{\mathcal B}

\newcommand{\conv}{\,\star\,}
\newcommand{\convp}{\,\star_p\,}

\newcommand{\Hm}{m_\group}
\newcommand{\dH}{\di{\Hm}}

\newcommand{\emb}[1]{\ode_{#1}}
\newcommand{\Emb}{\Ode}

\newcommand{\aux}{\gamma}
\newcommand{\auxn}{\aux_n}

\newcommand\indicator{1}

\newcommand{\leftreg}{\pi}

\begin{document}


\title [Lattice homomorphisms in harmonic analysis]{Lattice homomorphisms in harmonic analysis}

\author{H.\ Garth Dales}
\address{H.\ Garth Dales, Department of Mathematics and Statistics, University of Lancaster, Lancaster LA1 4YF, United Kingdom}
\email{g.dales@lancaster.ac.uk}

\author{Marcel de Jeu}
\address{Marcel de Jeu, Mathematical Institute, Leiden University, P.O.\ Box 9512, 2300 RA Leiden, the Netherlands}
\email{mdejeu@math.leidenuniv.nl}

\dedicatory{Dedicated to Ben de Pagter on the occasion of his 65th birthday}





\begin{abstract}
Let $S$ be a non-empty, closed subspace of a locally compact group $G$ that is a subsemigroup of $G$. Suppose that $X, Y$, and $Z$ are Banach lattices that are vector sublattices of the order dual $\textup{C}_{\textup{c}}(S,\mathbb R)^\sim$ of the real-valued, continuous functions with compact support on $S$, and where $Z$ is Dedekind complete. Suppose that $\ast: X\times Y\to Z$ is a positive bilinear map such that $\supp{(x\ast y)}\subseteq\supp{x}\,\cdot\,\supp{y}$ for all $x\in X^+$ and $y\in Y^+$ with compact support. We show that, under mild conditions, the canonically associated map from $X$ into the vector lattice of regular operators from $Y$ into $Z$ is then a lattice homomorphism. Applications of this result are given in the context of convolutions, answering questions previously posed in the literature.\\
As a preparation, we show that the order dual of the continuous, compactly supported functions on a closed subspace of a locally compact space can be canonically viewed as an order ideal of the order dual of the continuous, compactly supported functions on the larger space.\\
As another preparation, we show that $\textup{L}^p$-spaces and Banach lattices of measures on a locally compact space can be embedded as vector sublattices of the order dual of the continuous, compactly supported functions on that space.
\end{abstract}
\maketitle



\section{Introduction and overview}\label{sec:introduction_and_overview}

\noindent
Let $\group$ be a locally compact group with (real) measure algebra $\measgroupR$. Then $\measgroupR$ is not only a Banach algebra with convolution as multiplication, but also a Banach lattice. The left regular representation $\rep$ of $\measgroupR$ is easily seen to take its values in the algebra of regular operators $\regularops{\measgroupR}$ on $\measgroupR$, so that we actually have an algebra homomorphism $\rep:\measgroupR\to\regularops{\measgroupR}$. Furthermore, $\measgroupR$ is Dedekind complete, so that $\regularops{\measgroupR}$ is a vector lattice again. Hence it is meaningful to wonder whether the left regular representation  $\rep:\measgroupR\to\regularops{\measgroupR}$ is not only an algebra homomorphism, but also a lattice homomorphism. This question was raised during a workshop on ordered Banach algebras at the Lorentz Center in Leiden in 2014, and it occurs in Wickstead's list of open problems based on those that were posed during this workshop; see \cite{wickstead:2017c}.

The natural approach to this question is to start with one of the Riesz--Kantorovich formulae as a basis to determine whether $\rep$ is a lattice homomorphism, and to use the explicit formula for the convolution of two measures while doing so. Then the expressions become  complicated very quickly, and an answer has not been obtained along these lines so far.

Nevertheless, the answer to the question is known: the left regular representation $\rep:\measgroupR\to\regularops{\measgroupR}$ is indeed a lattice homomorphism. The first proof of this, as obtained by the present authors, is surprisingly simple. It uses just a little more than the fact that the support of the convolution of two measures with compact support is contained in the products of the support, combined with the general fact that the modulus on a vector lattice is additive on finite sums of mutually disjoint elements. The Riesz--Kantorovich formulae and the explicit expression for the convolution of two measures are not needed.

A closer look at the proof showed that, in fact, it does not really use that the objects involved are measures. Essentially the same proof establishes that, for $1\leq p<\infty$, the natural action of $\EllonegroupR$ on $\EllpgroupR$ by convolution gives a lattice homomorphism from $\EllonegroupR$ into the regular operators $\regularops{\EllpgroupR}$ on $\EllpgroupR$. In fact, under mild conditions, it shows that, `whenever' a Banach lattice $X$ on $\group$ convolves a Banach lattice $Y$ on $\group$ into a Dedekind complete Banach lattice $Z$ on $\group$, then the natural map from $X$ into the regular operators from $Y$ into $Z$ is a lattice homomorphism.
A still closer look showed that it is not even necessary that the action of $X$ on $Y$ be given by convolution. As long as it is a positive map that satisfies the property for supports mentioned above, essentially the same proof as for $\measgroupR$ shows that the natural map from $X$ into the regular operators from $Y$ into $Z$ is still a lattice homomorphism. As a rule of thumb, this is `always' true for convolution-like positive bilinear maps. Exaggerating a little, one could say that the main problem with the original question for $\measgroupR$ is that there is too much information that obscures the underlying picture.

Above, we have spoken loosely about `essentially the same proof' and `Banach lattices on $\group$'. It is evidently desirable to be able to make this precise, and then\textemdash hopefully\textemdash give the `essential' proof of one central theorem that clarifies the mathematical backbone of the situation, and that specialises to various practical cases of interest. This is, indeed, possible. As will become apparent, the order dual $\contcgroupRod$ of the continuous functions with compact support on $\group$ can act as a large vector lattice that\textemdash this is true in a more general context of locally compact spaces\textemdash contains various familiar Banach lattices as vector sublattices. It is in this framework that such a central theorem can, indeed, be established `once and for all'. The ensuing result, which is the group case of \cref{res:general_result_for_semigroups}, below, is the heart of this article.

\smallskip

There are many examples of Banach algebras on a locally compact \emph{semi}group $S$, provided with a convolution-like product, that are also Dedekind complete Banach lattices. Again, one can ask whether the left regular representation of these algebras is a lattice homomorphism. More generally again, if a Banach lattice $X$ on $S$ `convolves' a Banach lattice $Y$ on $S$ into a Banach lattice $Z$ on $S$, where $Z$ is Dedekind complete, is the canonically associated map from $X$ into the regular operators from $Y$ into $Z$ then a lattice homomorphism? Unfortunately, the proof of the general theorem as for groups is then no longer valid. Results can still be obtained, however, when one supposes that $S$ is actually a closed subset of a locally compact group $\group$. It is then possible to reduce the problem for $S$ to the problem for $\group$, where the answer is known. For this, one merely needs to be able to view Banach lattices that are sublattices of $\od{\contc(S,\RR)}$ as Banach lattices that are sublattices of $\contcgroupRod$. This is indeed possible, since\textemdash this is a special case of a general result for closed subspaces of locally compact spaces\textemdash it can be shown that one can canonically embed $\od{\contc(S,\RR)}$ as a vector sublattice of $\contcgroupRod$, with supports being preserved under the embedding. It is thus that the group case of our main result, \cref{res:general_result_for_semigroups}, below, can actually be used to establish a similar result for semigroups that are closed subsets of locally compact groups. In the end, the original result for locally compact groups (where the actual key proof can be given) is then a special case of \cref{res:general_result_for_semigroups}. This final result is described in the abstract of this article.

\smallskip

It may have become obvious from the above discussion that the present article is at the interface of the fields of positivity, abstract harmonic analysis, and Banach algebras. It is, perhaps, not yet very common to be familiar with the basic notions of these three disciplines together. It is for this reason that we have decided to explain the necessary terms and to review the necessary results from each of these fields in an attempt to make this article accessible to all readers, regardless of their background. We also hope that, by doing this, we shall facilitate further research at the junction of these disciplines.

\smallskip

This article is organised as follows.

\smallskip

Section~\ref{sec:vector_lattices_and_banach_lattices} contains basic notions and results for vector lattices and Banach lattices,

Banach lattices can be complexified to yield complex Banach lattices; this is the topic of Section~\ref{sec:complex_banach_lattices}.

Section~\ref{sec:banach_lattice_algebras} covers the basic notions of Banach algebras and Banach lattice algebras, and introduces complex Banach lattice algebras.

Section~\ref{sec:locally_compact_spaces} is concerned with locally compact spaces, and notably with the order dual $\contctsRod$ of the continuous, compactly supported functions on a locally compact space $\ts$. As will be explained in that section, this order dual is Bourbaki's space of Radon measures on $\ts$ as in \cite{bourbaki_INTEGRATION_VOLUME_I_CHAPTERS_1-6_SPRINGER_EDITION:2004}.

Section~\ref{sec:closed_subspaces_of_locally_compact_spaces} shows how the order dual $\contctstwoRod$ for a closed subspace $\tstwo$ of a locally compact space $\ts$ can be embedded into $\contctsRod$ as an order ideal. The reader whose interest lies in groups and not in semigroups can omit this section in its entirety. We are not aware of a reference for the results in this section, which may also find applications elsewhere.

Let $\ts$ be a locally compact space. As explained above in the context where $\ts$ is a locally compact group, it is necessary to embed various familiar Banach lattices on $\ts$ as vector sublattice of $\contctsRod$. This is done in Section~\ref{sec:embedding_familiar_vector_lattices}.  We are not aware of earlier results in this direction, where the r{\^o}le of $\contctsRod$ is not dissimilar to that of the space of distributions on an open subset of $\RR^d$ in the sense of Schwartz.

Section~\ref{sec:locally_compact_groups} contains the necessary material on locally compact groups and on Banach lattices and Banach lattice algebras on such groups.

Section~\ref{sec:locally_compact_semigroups} is of a similar nature as Section~\ref{sec:locally_compact_groups}, but now for semigroups. Taken together, Sections~\ref{sec:locally_compact_groups} and~\ref{sec:locally_compact_semigroups} contain a good stockpile of Banach lattice algebras. Some of them are semisimple while others are radical\textemdash this does not seem to influence the order properties of the left regular representations. We hope that these examples can also serve as test cases for further study of Banach lattice algebras in general.	

Section~\ref{sec:main_theorem} contains our key results. This section is the core of the present article and the other sections are, in a sense, merely auxiliary. The reader may actually wish to have a look at this section, and notably at the proof for the group case of \cref{res:general_result_for_semigroups}, before reading other sections.

In Section~\ref{sec:lattice_homomorphisms_in_harmonic_analysis}, all is put together. The general results from Section~\ref{sec:main_theorem}, combined with the embedding results from Section~\ref{sec:embedding_familiar_vector_lattices}, are now easily combined to yield that various canonical maps are actually lattice homomorphism. The left regular representation of $\meas(\group,\RR)$ is one of them. We also include in this section a list of cases where it is known whether the left regular representation of a Dedekind complete Banach lattice algebra is a lattice homomorphism or not.

Section~\ref{sec:possible_further_research_in_ordered_harmonic_analysis} discusses the relation between one of the results in Section~\ref{sec:lattice_homomorphisms_in_harmonic_analysis} and earlier work by Arendt, Brainerd and Edwards, and Gilbert.  This leads to questions for further research, on which we hope to be able to report in the future.

\smallskip

We conclude this section by introducing a few conventions and notations.

\smallskip

The vector spaces and algebras in this article are all over the real field, $\RR$, unless stated otherwise. This is the canonical convention in the field of positivity. On the other hand, the canonical convention in the context of Banach algebras and abstract harmonic analysis is that the base field be the complex field, $\CC$. There seems to be no natural way to reconcile these two conventions where these disciplines meet. In view of the prominent r{\^o}le of ordering in the present article, we have chosen to consistently side with the convention in positivity. Readers from a different background are, therefore, cautioned to realise that a Banach algebra is a real Banach algebra, and that, e.g., the measure algebra of a locally compact group consists of the \emph{real} signed regular Borel measures on the group.  We apologise for the mental dissonance that such consequences of our efforts to be precise and consistent will almost inevitably cause. In a further attempt to prevent misunderstanding as much as possible, we have included the field in the notation for concrete spaces. The group algebra of a locally compact group is denoted by $\EllonegroupR$, for example.

We shall let $\FF$ denote the choice for either $\RR$ or $\CC$ when results are valid in both cases.

Algebras are always linear and associative. An algebra need not have an identity element. An algebra homomorphism between two unital algebras need not map the identity element to the identity element.

Topological spaces are always supposed to be Hausdorff, unless stated otherwise.

Let $\ts$ be a topological space. Then we let $\conttsR$ denote the real-valued, continuous functions on $\ts$, we let $\contbtsR$ denote the real-valued, bounded, continuous functions on $\ts$, we let $\contotsR$ denote the real-valued, continuous functions on $\ts$ that vanish at infinity, and we let $\contctsR$ denote the real-valued, continuous functions on $\ts$ with compact support. Their complex counterparts $\conttsC$, $\contbtsC$, $\contotsC$, and $\contctsC$ are similarly defined.

Let $S$ be a non-empty set. Then $\norm{f}_\infty$ denotes the uniform norm of a bounded, real- or complex-valued function $f$ on $S$. Sometimes we shall write  $\norm{f}_{\infty,S}$ if confusion could arise otherwise.

Let $\vl$ and $\vltwo$ be normed spaces over $\FF$. Then $\bounded(\vl,\vltwo)$ denotes the bounded linear operators from $\vl$ into $\vltwo$. We shall write $\bounded(\vl)$ for $\bounded(\vl,\vl)$.

The identity element of a group $\group$ is denoted by $e_G$.

Semigroups need not have identity elements.

Let $S$ be a semigroup, and suppose that $A_1$ and $A_2$ are non-empty subsets of $S$. Then we set $A_1\,\cdot\,A_2 \coloneqq \desset{ a_1a_2: a_1\in A_1,\,a_2\in A_2}$.

\section{Vector lattices and Banach lattices}\label{sec:vector_lattices_and_banach_lattices}

\noindent In this section, we shall cover some basic material on vector and Banach lattices. The details can be found in introductory books such as \cite{de_jonge_van_rooij_INTRODUCTION_TO_RIESZ_SPACES:1977,zaanen_INTRODUCTION_TO_OPERATOR_THEORY_IN_RIESZ_SPACES:1997}. More advanced general references are  \cite{abramovich_aliprantis_INVITATION_TO_OPERATOR_THEORY:2002,abramovich_aliprantis_PROBLEMS_IN_OPERATOR_THEORY:2002,aliprantis_burkinshaw_LOCALLY_SOLID_RIESZ_SPACES_WITH_APPLICATIONS_TO_ECONOMICS_SECOND_EDITION:2003,aliprantis_burkinshaw_POSITIVE_OPERATORS_SPRINGER_REPRINT:2006,luxemburg_zaanen_RIESZ_SPACES_VOLUME_I:1971,meyer-nieberg_BANACH_LATTICES:1991,schaefer_BANACH_LATTICES_AND_POSITIVE_OPERATORS:1974,wnuk_BANACH_LATTICES_WITH_ORDER_CONTINUOUS_NORMS:1999,zaanen_RIESZ_SPACES_VOLUME_II:1983}.

Suppose that $\vl$ is a partially ordered vector space, i.e., a vector space that is supplied with a partial ordering such that $x+z\geq y+z$ for all $z\in\vl$ whenever $x,y\in\vl$ are such that $x\geq y$, and such that $\alpha x\geq 0$ whenever $x\geq 0$ in $\vl$ and $\alpha\geq 0$ in $\RR$. The subset of positive elements of $\vl$ is then a cone, and it is denoted by $\pos{E}$.

A \emph{vector lattice} or \emph{Riesz space} is a partially ordered vector space $\vl$ such that every two elements $x,y$ of $\vl$ have a least upper bound in $\vl$; this supremum of the set $\set{x,y}$ is denoted by $x\vee y$. The infimum of $\set{x,y}$ then also exists; it is denoted by $x\wedge y$. For $x\in\vl$, we define its \emph{modulus $\abs{x}$} as $\abs{x}\coloneqq x\vee(-x)$, its \emph{positive part $\pos{x}$} as $\pos{x}\coloneqq x\vee 0$, and its \emph{negative part $\nega{x}$} as $\nega{x}\coloneqq (-x)\vee 0$. Then $\pos{x},\nega{x}\in\pos{E}$, $x=\pos{x}-\nega{x}$, and $\abs{x}=\pos{x}+\nega{x}$.

Let $\vl$ be a vector lattice. Two elements $x$ and $y$ of $\vl$ are \emph{disjoint} if $\abs{x}\wedge\abs{y}=0$; this is denoted by $x\perp y$. When this is the case, then $\abs{x+y}=\abs{x}+\abs{y}$. This latter property lies at the heart of the results in this article, and can be found in \cite[Theorem~14.4(i)]{luxemburg_zaanen_RIESZ_SPACES_VOLUME_I:1971} and \cite[Theorem~8.2(i)]{zaanen_INTRODUCTION_TO_OPERATOR_THEORY_IN_RIESZ_SPACES:1997}, for example.

Let $x\in\vl$. Then $\pos{x}\perp\nega{x}$. Suppose that $x=y_1-y_2$ with $y_1,y_2\in\pos{\vl}$. Then $y_1\geq \pos{x}$ and $y_2
\geq\nega{x}$. Suppose, further, that $y_1\perp y_2$. Then $y_1=\pos{x}$ and $y_2=\nega{x}$.

Let $\vl$ be a vector lattice, and let $\vltwo$ be a linear subspace of $\vl$. Then $\vltwo$ is a \emph{vector sublattice of $\vl$} if $x\vee y\in\vltwo$ whenever $x,y\in\vltwo$; then also $x\wedge y\in\vltwo$ whenever $x,y\in\vltwo$, and $\abs{x}\in\vltwo$ whenever $x\in\vltwo$.

Let $\vl$ be a vector lattice, and let $\vltwo$ be a vector sublattice of $\vl$. Then $\vltwo$ is an \emph{order ideal of $\vl$} if $x\in\vltwo$ whenever $x,y\in\vl$ are such that $\abs{x}\leq\abs{y}$ and $y\in\vltwo$.

An \emph{order interval} in a vector lattice $\vl$ is a subset of the form 
\[
\desset{x\in\vl : a\leq x\leq b}
\] 
for some $a\leq b$ in $\vl$. A subset of $\vl$ is \emph{order bounded} if it is contained in an order interval.

A vector lattice $\vl$ is \emph{Dedekind complete} or \emph{order complete} if every non-empty subset of $\vl$ that is bounded above in $\vl$ has a supremum in $\vl$.

\begin{example}\label{ex:spaces_of_continuous_functions_one}
	Let $\ts$ be a non-empty, topological space. Then $\cont{(\ts,\RR)}$, $\contb{(\ts,\RR)}$, $\contotsR$, and $\contctsR$ are vector lattices when supplied with the pointwise ordering.
	
	Let $\ts$ be a non-empty, compact space. Then $\conttsR$ is Dedekind complete if and only if $\ts$ is extremely disconnected (some sources write `extremally disconnected'), i.e., if and only if the closure of every open subset of $\ts$ is open. This result is due to Nakano; see \cite[Proposition~4.2.9]{dales_BANACH_ALGEBRAS_AND_AUTOMATIC_CONTINUITY:2000},  \cite[Theorem~2.3.3]{dales_dashiell_lau_strauss_BANACH_SPACES_OF_CONTINUOUS_FUNCTIONS_AS_DUAL_SPACES:2016}, or  \cite[Theorem~12.16]{de_jonge_van_rooij_INTRODUCTION_TO_RIESZ_SPACES:1977}, for example. The Stone--\v{C}ech compactification $\beta\NN$ of the natural numbers $\NN$ is an example of a compact, extremely disconnected space.
	
\end{example}

\begin{example}\label{ex:Ellp_spaces_one}
	Let $\ts$ be a non-empty set, let $\borelsalg$ be a $\sigma$-algebra of subsets of $\ts$, and let $\mu:\borelsalg\to[0,\infty]$ be a measure on $\borelsalg$. For $1\leq p\leq\infty$, we supply $\Ell^p(\ts,\borelsalg,\mu,\RR)$ with the pointwise $\mu$-almost everywhere partial ordering. Then $\Ell^p(\ts,\borelsalg,\RR)$ is a vector lattice. For $1\leq p<\infty$, it is Dedekind complete. For $p=\infty$, it is Dedekind complete if $\mu$ is localisable, i.e., if every measurable subset of $\ts$ of infinite measure has a measurable subset of finite,  strictly positive measure and the measure algebra of $\ts$ is order complete. In particular, $\Ell^\infty(\ts,\borelsalg,\mu)$ is Dedekind complete when $\mu$ is $\sigma$-finite. We refer to \cite[p.~126-127]{luxemburg_zaanen_RIESZ_SPACES_VOLUME_I:1971} and \cite[Definition~211G, Theorem~211L, and Theorem~243H]{fremlin_MEASURE_THEORY_VOLUME_2:2003} for proofs.
	
	An example, taken from \cite{troitsky_UNPUBLISHED:2017}, where $\Ell^\infty(\ts,\borelsalg,\mu)$ is not Dedekind complete, is as follows. Let $\ts$ be an uncountable set, and let $\borelsalg$ be the $\sigma$-algebra of all subsets $A$ of $\ts$ such that either $A$ or $\ts\setminus A$ is uncountable. Let $\mu$ be the counting measure on $\borelsalg$. Take a subset $U$ of $\ts$ such that both $U$ and $\ts\setminus U$ are uncountable, and set
	\[
	S\coloneqq\desset{\indicator_A: A\subset U\text{ and }A\text{ is countable} }.
	\]
	The $S$ is a subset of $\Ell^\infty(\ts,\borelsalg,\mu)$ that is bounded above, but $S$ has no supremum in $\Ell^\infty(\ts,\borelsalg,\mu)$. Hence $\Ell^\infty(\ts,\borelsalg,\mu)$ is not Dedekind complete.
	
	\end{example}

\begin{example}\label{ex:spaces_of_measures_one}
	Let $\ts$ be a non-empty set, and let $\borelsalg$ be a $\sigma$-algebra of subsets of $\ts$. We let $\meas(\ts,\borelsalg,\RR)$ be the vector space of all signed measures $\mu:\borelsalg\to\RR$. We introduce a partial ordering on $\meas(\ts,\borelsalg,\RR)$ by setting $\mu\geq\nu$ whenever $\mu,\nu\in \meas(\ts,\borelsalg,\RR)$ are such that $\mu(A)\geq\nu(A)$ for all $A\in\borelsalg$. Then $\meas(\ts,\borelsalg,\RR)$ is a Dedekind complete vector lattice; see \cite[p.~187]{zaanen_INTRODUCTION_TO_OPERATOR_THEORY_IN_RIESZ_SPACES:1997}. For $\mu,\nu\in\meas(\ts,\borelsalg,\RR)$, the supremum $\mu\vee\nu$ of $\mu$ and $\nu$ is given by the formula
	\begin{equation}\label{eq:supremum_of_two_measures}
	(\mu\vee\nu)(A)=\sup\desset{\mu(B)+\nu(A\setminus B) : B\in\borelsalg,\,B\subseteq A}
	\end{equation}
	for $A\in\borelsalg$. The formula for the infimum is similar, and,  for $\mu\in\meas(\ts,\borelsalg,\RR)$, we have
	\begin{equation}\label{eq:modulus_of_a_measure}
	\abs{\mu}(A)=\sup\lrdesset{\sum_{i=1}^n \abs{\mu(B_i)}:B_1,\ldots,B_n\in\borelsalg\text{ form a disjoint partition of }A}
	\end{equation}
	for $A\in\borelsalg$. That is, $\abs{\mu}$ is the usual total variation measure of $\mu$.
\end{example}

Suppose that $\vl$ and $\vltwo$ are vector lattices and that $T:\vl\to\vltwo$ is a linear operator. Then $T$ is \emph{order bounded} if $T$ maps order bounded subsets of $\vl$ to order bounded subsets of $\vltwo$. Equivalently, $T$ should map order intervals in $\vl$ into order intervals in $\vltwo$. The order bounded linear operators from $\vl$ into $\vltwo$ form a vector space that is denoted by $\orderboundedops{\vl,\vltwo}$. We shall write $\orderboundedops{\vl}$ for $\orderboundedops{\vl,\vl}$.

Let $S,T:\vl\to\vltwo$ be order bounded linear operators. Then we say that $S\geq T$ if $Sx\geq Tx$ for all $x\in\pos{\vl}$. This introduces a partially ordering on $\orderboundedops{\vl,\vltwo}$. The \emph{regular operators} from $\vl$ into $\vltwo$ are the elements of the subspace $\regularops{\vl,\vltwo}$ of $\orderboundedops{\vl,\vltwo}$ that is spanned by the positive linear operators from $\vl$ into $\vltwo$. Thus the regular operators from $\vl$ into $\vltwo$ are the linear operators $T$ from $\vl$ into $\vltwo$ that can be written as $T=S_1-S_2$, where $S_1,S_2\in\orderboundedops{\vl,\vltwo}$ are both positive. We shall write $\regularops{\vl}$ for $\regularops{\vl,\vl}$.

It is not generally true that the partially ordered vector spaces $\orderboundedops{\vl,\vltwo}$ or $\regularops{\vl,\vltwo}$ are again vector lattices, but there is a sufficient condition on the codomain for this to be the case. We have the following; see \cite[Theorem~1.18]{aliprantis_burkinshaw_POSITIVE_OPERATORS_SPRINGER_REPRINT:2006} or \cite[Theorem~20.4]{zaanen_INTRODUCTION_TO_OPERATOR_THEORY_IN_RIESZ_SPACES:1997}, for example.

\begin{theorem}\label{res:riesz_kantorovich}
	Let $\vl$ and $\vltwo$ be vector lattices such that $\vltwo$ is Dedekind complete. Then the spaces $\orderboundedops{\vl,\vltwo}$ and $\regularops{\vl,\vltwo}$ coincide. Moreover, $\regularops{\vl,\vltwo}$ is a Dedekind complete vector lattice, where the lattice operations are given by
	\begin{align}
	\abs{T}(x)&=\sup\desset{\abs{Ty}:\abs{y}\leq x},\label{eq:RK1}\\
	[S\vee T](x)&=\sup\desset{Sy+Tx: y,z\in\pos{\vl},\,y+z=x}, \text{ and}\label{eq:RK2}\\
	[S\wedge T](x)&=\inf\desset{Sy+Tx: y,z\in\pos{\vl},\,y+z=x}\label{eq:RK3}
	\end{align}
	for all $S,T\in\regularops{\vl,\vltwo}$ and $x\in\pos{\vl}$.
	
\end{theorem}

The formulae in the above theorem are the \emph{Riesz--Kantorovich formulae}.

Applying the theorem with $\vltwo=\RR$, we see that the order bounded linear functionals on $\vl$ coincide with the regular ones, and that they form a vector lattice. This vector lattice is denoted by $\vlod$, and it is called the \emph{order dual of $\vl$}. Of course, for $\ode\in\vlod$, we have $\ode\geq 0$ if and only if $\pairing{\ode,x}\geq 0$ for all $x\in\pos{E}$.

Suppose that $\vl$ and $\vltwo$ are vector lattices. A linear operator $T:\vl\to\vltwo$ is a \emph{lattice homomorphism} if $T(x\vee y)=Tx\vee Ty$ for all $x,y\in\vl$. This is equivalent to requiring that $T(x\wedge y)=Tx\wedge Ty$ for all $x,y\in\vl$, and also equivalent to requiring that $\abs{Tx}=T\abs{x}$ for all $x\in\vl$. Lattice homomorphisms are positive linear operators.

A linear operator $T:\vl\to\vltwo$ is \emph{interval preserving} if it is positive and such that $T([0,x])=[0,Tx]$ for all $x\in\pos{\vl}$. The positivity of $T$ already implies that $T([0,x])\subseteq[0,Tx]$; the point is that equality should hold.

Let $T:\vl\to\vltwo$ be an order bounded linear operator. Then its \emph{order adjoint} $\oadj{T}:\od{\vltwo}\to\vlod$ is defined by setting
\[
\pairing{\oadj{T}\ode,x}\coloneqq \pairing{\ode,Tx}
\]
for $x\in\vl$ and $\ode\in\vlod$. In Section~\ref{sec:locally_compact_spaces}, we shall use the following two results; see \cite[Theorems~2.19 and~2.20]{aliprantis_burkinshaw_POSITIVE_OPERATORS_SPRINGER_REPRINT:2006}.

\begin{proposition}\label{res:dual_is_lattice_homomorphism}
	Let $T:\vl\to\vltwo$ be an interval preserving linear operator between the vector lattices $\vl$ and $\vltwo$. Then $\oadj{T}:\od{\vltwo}\to\vlod$ is a lattice homomorphism.
\end{proposition}

\begin{proposition}\label{res:dual_is_interval_preserving}
	Let $T:\vl\to\vltwo$ be a positive linear operator between the vector lattices $\vl$ and $\vltwo$, where $\vltwo$ is such that $\od{\vltwo}$ separates the points of $\vltwo$. Then $T$ is a lattice homomorphism if and only if $\oadj{T}:\od{\vltwo}\to\vlod$ is interval preserving.
\end{proposition}

Let $\vl$ be a vector lattice. Then a norm $\norm{\,\cdot\,}$ on $\vl$ is a \emph{lattice norm} if $\norm{x}\leq\norm{y}$ whenever $x$ and $y$ in $\vl$ are such that $\abs{x}\leq\abs{y}$.

\begin{definition} A Banach space $(\vl,\norm{\,\cdot\,})$ for which $\vl$ is a vector lattice and $\norm{\,\cdot\,}$ is a lattice norm is a
	\emph{Banach lattice}.
\end{definition}

\begin{example}\label{ex:spaces_of_continuous_functions_two}
	Let $\ts$ be a topological space. Then the vector lattices $\contb{(\ts,\RR)}$ and $\conto(\ts,\RR)$ from \cref{ex:spaces_of_continuous_functions_one} are Banach lattices when supplied with the uniform norm $\norm{\,\cdot\,}_\infty$.
\end{example}

\begin{example}\label{ex:Ellp_spaces_two}
	Let $\ts$ be a non-empty set, let $\borelsalg$ be a $\sigma$-algebra of subsets of $\ts$, and let $\mu:\borelsalg\to[0,\infty]$ be a measure on $\borelsalg$. Then the vector lattices $\Ell^p(\ts,\borelsalg,\mu,\RR)$ from \cref{ex:Ellp_spaces_one}  are Banach lattices when supplied with the usual $p$-norm $\norm{\,\cdot\,}_p$.
\end{example}

\begin{example}\label{ex:spaces_of_measures_two}
	Let $\ts$ be a non-empty set, and let $\borelsalg$ be a $\sigma$-algebra of subsets of $\ts$. Then the vector lattice $\meas(\ts,\borelsalg,\RR)$ of real-valued measures on $\borelsalg$ from \cref{ex:spaces_of_measures_one} is a Banach lattice when supplied with the norm $\norm{\,\cdot\,}$ that is obtained by setting
	\begin{equation}\label{eq:norm_of_a_measure}
	\norm{\mu}\coloneqq\abs{\mu}(\ts).	
	\end{equation}
\end{example}

Let $\vl$ be a Banach lattice. Then $\vl$ has an order dual $\vlod$ as a vector lattice, as well as a topological dual $\vlnd$ as a Banach space. It is a fundamental fact that $\vlod=\vlnd$; see \cite[Corollary~4.4]{aliprantis_burkinshaw_POSITIVE_OPERATORS_SPRINGER_REPRINT:2006} or \cite[Theorem~25.8(iii)]{zaanen_INTRODUCTION_TO_OPERATOR_THEORY_IN_RIESZ_SPACES:1997}, for example.

Suppose that $\vl$ is a Banach lattice, that $\vltwo$ is a normed vector lattice, and that the map $T:\vl\to\vltwo$ is an order bounded linear operator. Then $\vl$ is automatically continuous; see \cite[Theorem~4.3]{aliprantis_burkinshaw_POSITIVE_OPERATORS_SPRINGER_REPRINT:2006}, for example. In the sequel we shall repeatedly use the special case that a positive linear operator from a Banach lattice into a normed vector lattice is automatically continuous.

Let $\vl$ and $\vltwo$ be Banach lattices, where $\vltwo$ is Dedekind complete. Then we know from \cref{res:riesz_kantorovich} that $\regularops{\vl,\vltwo}$ is a Dedekind complete vector lattice. It can be supplied with the operator norm, but this is not generally a lattice norm. One can, however, define the \emph{regular norm} $\regularnorm{\,\cdot\,}$ on $\regularops{\vl,\vltwo}$ by setting \[
\regularnorm{T}\coloneqq\norm{\abs{T}}
\]
for $T\in\regularops{\vl,\vltwo}$. The regular norm is a lattice norm on $\regularops{\vl,\vltwo}$, and $\regularops{\vl,\vltwo}$ is then a Dedekind Banach lattice; see \cite[Theorem~4.74]{aliprantis_burkinshaw_POSITIVE_OPERATORS_SPRINGER_REPRINT:2006}, for example.

\section{Complex Banach lattices}\label{sec:complex_banach_lattices}

\noindent
In abstract harmonic analysis, Banach spaces and Banach algebras are almost always over the complex numbers. It is for this reason that we include the following material on complex Banach lattices. Details can be found in \cite[Section~3.2]{abramovich_aliprantis_INVITATION_TO_OPERATOR_THEORY:2002}, \cite[Section~2.2]{meyer-nieberg_BANACH_LATTICES:1991}, or
\cite[Section~2.11]{schaefer_BANACH_LATTICES_AND_POSITIVE_OPERATORS:1974}, for example.

Let $\vl$ be a Banach lattice. Then its complexified vector space $\compvl$ can be supplied with a modulus $\compabs{\,\cdot\,}:\compvl\to\vl$. The definition of $\compabs{\,\cdot\,}$ is analogous to one of the possible descriptions of the modulus of a complex number, as follows. For $x,y\in\vl$, the supremum
\[
\sup\desset{\Re(\el^{\il\theta}(x+\il y)):0\leq\theta\leq 2\pi}=\sup\desset{x \cos\theta+y\sin\theta:0\leq\theta\leq 2\pi}
\]
can be shown to exist in $\vl$, and we define this supremum to be the modulus $\compabs{x+\il y}$ of the element $x+\il y$ of $\compvl$. Then $\compabs{\,\cdot\,}$ extends the modulus $\abs{\,\cdot\,}$ on $\vl$. Take $z\in\compvl$. Then $\compabs{z}=0$ if and only if $z=0$. Furthermore, $\compabs{\alpha z}=\abs{\alpha}\compabs{z}$ for all $\alpha\in\CC$ and $z\in\compvl$, and $\compabs{w+z}\leq\compabs{w}+\compabs{z}$ for all $w,z\in\compvl$.

Set $\compnorm{z}\coloneqq\norm{\compabs{z}}$ for $z\in\compvl$. Then $\compnorm{\,\cdot\,}$ is a norm on $\compvl$ that extends the norm on $\vl$, and $(\compvl,\compnorm{\,\cdot\,})$ is a complex Banach space that is called a \emph{complex Banach lattice}.  As a topological vector space, $\compvl$ is $\RR$-linearly homeomorphic to the Cartesian product $\vl\times\vl$. One of the things to remember is that the non-zero complex Banach lattices are not lattices: they do have a modulus, but there is no r{\^o}le for a partial ordering on $\compvl$ as a whole.

\begin{example}\label{ex:spaces_of_continuous_functions_three}
	Let $\ts$ be a topological space. Then the complexifications of the Banach lattice $\contb{(\ts,\RR)}$, respectively, $\conto(\ts,\RR)$, from \cref{ex:spaces_of_continuous_functions_two} can be identified with the Banach space $\contb{(\ts,\CC)}$, respectively, $\conto(\ts,\CC)$, with the usual pointwise complex modulus and with the uniform norm $\norm{\,\cdot\,}_\infty$.
\end{example}

\begin{example}\label{ex:Ellp_spaces_three}
	Let $\ts$ be a non-empty set, let $\borelsalg$ be a $\sigma$-algebra of subsets of $\ts$, and let $\mu:\borelsalg\to[0,\infty]$ be a measure on $\borelsalg$. Then the complexifications of the Banach lattices $\Ell^p(\ts,\borelsalg,\mu,\RR)$ from \cref{ex:Ellp_spaces_two} can be identified with the Banach spaces $\Ell^p(\ts,\borelsalg,\mu,\CC)$, with the usual pointwise $\mu$-almost everywhere complex modulus and with the usual $p$-norm $\norm{\,\cdot\,}_p$.
\end{example}

\begin{example}\label{ex:spaces_of_measures_three}
	Let $\ts$ be a non-empty set, and let $\borelsalg$ be a $\sigma$-algebra of subsets of $\ts$. Then the complexification of the Banach lattice $\meas(\ts,\borelsalg,\RR)$ of real-valued measures on $\borelsalg$ from \cref{ex:spaces_of_measures_two} can be identified with the Banach space $\meas(\ts,\borelsalg,\CC)$ of complex-valued measures on $\borelsalg$, where the modulus, respectively, the norm, is again given by \cref{eq:modulus_of_a_measure}, respectively, \cref{eq:norm_of_a_measure}.
\end{example}

Let $\vl$ and $\vltwo$ be Banach lattices, and let $T:\vl\to\vltwo$ be a bounded linear operator. Then its complex-linear extension $\comp{T}:(\compvl,\compnorm{\,\cdot\,})\to(\compvltwo,\compnorm{\,\cdot\,})$ is a bounded linear operator, and $\norm{T}\leq\norm{\comp{T}}\leq2\norm{T}$. If $T\geq 0$, then $\norm{\comp{T}}=\norm{T}$.

Let $\compvl$ and $\compvltwo$ be complex Banach lattices. Then every complex-linear operator $T:\compvl\to\compvltwo$ has a unique expression as $T=S_1+\il S_2$, where $S_1,S_2:\vl\to\vltwo$ are real-linear operators, and
\[
(S_1+\il S_2)(x+\il y)=(Sx-Ty)+\il(Sy+Tx)
\]
for $x,y\in\vl$. Then $T$ is \emph{order bounded} (respectively, \emph{regular}) if both $S_1$ and $S_2$ are order bounded (respectively, regular). The complex vector space of all order bounded (respectively, regular) complex-linear operators from $\compvl$ into $\compvltwo$ is denoted by $\orderboundedops{\compvl,\compvltwo}$ (respectively, $\regularops{\compvl,\compvltwo}$). Then $\regularops{\compvl,\compvltwo}\subseteq\orderboundedops{\compvl,\compvltwo}\subseteq\bounded(\compvl,\compvltwo)$. A complex-linear operator $T:\compvl\to\compvltwo$ is \emph{positive} if $T(\pos{\vl})\subseteq\pos{\vltwo}$; this implies that $T(\vl)\subseteq\vltwo$. For such positive $T$, we have $\compabs{Tz}\leq T(\compabs{z})$ for $z\in\compvl$.
A complex-linear operator $T:\compvl\to\compvltwo$ is a \emph{complex lattice homomorphism} if $\compabs{Tz}=T\left(\compabs{z}\right)$ for all $z\in\compvl$. This is the case if and only if $T$ leaves $\vl$ invariant and the restricted map $T\mid_E:\vl\to\vl$ is a lattice homomorphism; see \cite[p.~136]{schaefer:1960}.

Let $\vl$ and $\vltwo$ be Banach lattices, where $\vltwo$ is Dedekind complete. Then the space $(\regularops{\vl,\vltwo},\regularnorm{\,\cdot\,})$ is a Dedekind complete Banach lattice, so that we can consider the complex Banach lattice $(\comp{[\regularops{\vl,\vltwo}]},\norm{\,\cdot\,}_{\tfs{r},\CC})$. For $T\in\ \comp{[\regularops{\vl,\vltwo}]}$, we have, by definition, that \[
\norm{T}_{\tfs{r},\CC}=\regularnorm{\compabs{T}}=\norm{\compabs{T}},
\]
and then the norm $\norm{\,\cdot\,}_{\tfs{r},\CC}$ on $\comp{[\regularops{\vl,\vltwo}]}$ extends the norm $\regularnorm{\,\cdot\,}$ on $\regularops{\vl,\vltwo}$. It is clear from the definitions that $\regularops{\compvl,\compvltwo}$ and $\comp{[\regularops{\vl,\vltwo}]}$ can be identified as complex vector spaces. Let $T\in\regularops{\compvl,\compvltwo}$. Then, viewing $T$ as an element of $\comp{[\regularops{\vl,\vltwo}]}$, so that $\compabs{T}$ is defined in $\comp{[\regularops{\vl,\vltwo}]}$, and viewing $\compabs{T}$ as an element of $\regularops{\compvl,\compvltwo}$ again, we have
\[
\compabs{T}x=\sup\lrdesset{\compabs{Tz}: z\in\compvl, \,\compabs{z}\leq x}
\]
for all $x\in\pos{\vl}$, and
\begin{equation}\label{eq:complex_inequality}
\compabs{Tz}\leq\compabs{T}\compabs{z}
\end{equation}
for all $z\in\compvl$.

Let $(\vl,\norm{\,\cdot\,})$ be a Banach lattice with dual Banach lattice $(\nd{\vl},\nd{\norm{\,\cdot\,}})$.
It follows from \cref{eq:complex_inequality} that the norm dual of the complex Banach lattice  $(\comp{\vl},\compnorm{\,\cdot\,})$ is canonically isometrically isomorphic as a complex Banach space to the complex Banach lattice $\left(\comp{\left(\nd{\vl}\right)},\comp{\left(\norm{\,\cdot\,}^\prime\right)}\right)$. In particular, analogously to the case of real scalars, the norm dual of a complex Banach lattice is again a complex Banach lattice.

\section{Banach algebras and Banach lattice algebras}\label{sec:banach_lattice_algebras}

\noindent
In this section, we shall review some material about Banach algebras, Banach lattice algebras, and their complex versions.

A \emph{Banach algebra} (respectively, a \emph{complex Banach algebra}) is a pair $(\alg,\norm{\,\cdot\,})$, where $\alg$ is an algebra (respectively, a complex algebra) with a norm $\norm{\,\cdot\,}$ such that $(\alg,\norm{\,\cdot\,})$ is a Banach space (respectively, a complex Banach space) and
\[
\norm{a_1a_2}\leq\norm{a_1}\norm{a_2}
\]
for $a_1,a_2\in\alg$. An identity element, if present, need not have norm 1. A net $(a_i)_{i\in I}$ in $\alg$ is an \emph{approximate identity} if $\lim_i a_ia=\lim_i aa_i=a$ for all $a\in\alg$. If, in addition, $\norm{a_i}\leq 1$ for all $i\in I$, then the approximate identity $(a_i)_{i\in I}$ is \emph{contractive}.

Let $\alg$ and $\algtwo$ be Banach algebras. Then a map $\rep:\alg\to\algtwo$ is a \emph{Banach algebra homomorphism} if it is a continuous algebra homomorphism. The notion of a \emph{complex Banach algebra homomorphism} between two complex Banach algebras is similarly defined.

For an introduction to the theory of complex Banach algebras, see \cite{allan_INTRODUCTION_TO_BANACH_SPACES_AND_ALGEBRAS:2011}, for example; a more substantial account is given in \cite{dales_BANACH_ALGEBRAS_AND_AUTOMATIC_CONTINUITY:2000}. As long as one does not move into topics where working over the complex field is manifestly essential\textemdash the latter actually constitute most of the theory\textemdash several of the (more basic) results about complex Banach algebras are obviously also true for Banach algebras.

Canonical examples of Banach algebras are $\bounded(\vl)$, where $\vl$ is a Banach space, and $\contbtsR$ and $\contotsR$, where $\ts$ is a topological space and where the norm on both algebras is the supremum-norm $\norm{\,\cdot\,}_\infty$. Examples of complex Banach algebras are obtained likewise.

In Section~\ref{sec:locally_compact_groups}, we shall give examples of Banach algebras and complex Banach algebras on locally compact groups that involve convolution.

\smallskip

Let $\alg$ be a complex algebra. A proper left ideal $I$ in $\alg$ is \emph{modular} if there exists  $u\in \alg$ with $a-au\in I$ for all $a\in \alg$.
The family of modular left ideals in $\alg$ (if non-empty) has maximal members, and the \emph{\uppars{Jacobson} radical of $\alg$} is the intersection of the maximal modular left ideals of $\alg$ \cite[Section~1.5]{dales_BANACH_ALGEBRAS_AND_AUTOMATIC_CONTINUITY:2000}; it is denoted by  ${\rm rad\,} \alg$, where we set ${\rm rad\,} \alg\coloneqq\alg$ when $\alg$ has no
maximal modular left ideals. In fact, ${\rm rad\,} \alg $ is a (two-sided)  ideal in $\alg$.  The complex algebra $\alg$ is \emph{semisimple} when  ${\rm rad\,} \alg=\{0\}$ and \emph{radical} when ${\rm rad\,} \alg=\alg$.

Let $\alg$ be a complex Banach algebra. Then ${\rm rad\,} \alg $ is closed in $\alg$, and $\alg/{\rm rad\,} \alg $ is a semisimple complex Banach algebra.  An element $a\in \alg$ is \emph{quasi-nilpotent} if $\lim_{n\to \infty} \norm{a^n}^{1/n}=0$.   Each quasi-nilpotent element belongs to ${\rm rad\,} \alg $,
and ${\rm rad\,} \alg $ is equal to the set of quasi-nilpotent elements in the special case that $\alg$ is commutative.

\smallskip

Banach lattice algebras combine the structures of Banach lattices and of Banach algebras. Their definition in the present article is as follows.

\begin{definition}
Let $\alg$ be a Banach lattice that is also a Banach algebra such that the product of two positive elements is again positive. Then $\alg$ is a \emph{Banach lattice algebra}.
\end{definition}

We note that the norm on a Banach lattice algebra is compatible with both the order and product.

There are further remarks concerning the definition of a Banach lattice algebra, in particular involving the r\^ole of an identity, in \cite{wickstead:2017b}. In the present article, we leave this unspecified: the algebra need not be unital, nor need an identity element, if present, be positive.

As compared to the general theory of Banach algebras or operator algebras the theory of Banach lattice algebras is largely undeveloped. We refer to \cite{wickstead:2017b,wickstead:2017c} for a survey and for open problems. Problems~6 and~7 in \cite{wickstead:2017c} are resolved by \cref{res:action_of_Ellone_on_Ellp,res:left_regular_representation_of_measures}, respectively, in the present article.

Let $\alg$ be a Banach lattice algebra, and take $a_1,a_2\in\alg$. By splitting each of $a_1$ and $a_2$ into their positive and negative parts, it follows easily that $\abs{a_1a_2}\leq\abs{a_1}\abs{a_2}$. This holds, in fact, in every so-called \emph{Riesz algebra}, i.e., in every vector lattice that is an algebra with the property that the product of two positive elements is again positive.

\begin{example}
Let $\ts$ be a topological space. Then $\contbtsR$ and $\contotsR$, with the uniform norm and pointwise ordering, are Banach lattice algebras.
\end{example}

\begin{example}
	Let $\vl$ be a Dedekind complete Banach lattice. Then $\regularops{\vl}$ is a Dedekind complete Banach lattice and also an algebra. It is, in fact, a Riesz algebra. Since then $\abs{T_1T_2}\leq\abs{T_1}\abs{T_2}$ for $T_1,T_2\in\regularops{\vl}$, it follows that the regular norm $\regularnorm{\,\cdot\,}$ is submultiplicative on $\regularops{\vl}$. Hence $(\regularops{\vl},\regularnorm{\,\cdot\,})$ is a Dedekind complete Banach lattice algebra.
\end{example}

In Section~\ref{sec:locally_compact_groups}, we shall define the group algebra and the measure algebra of a locally compact group. These Banach algebras are Banach lattice algebras.

\begin{definition}\label{def:banach_lattice_algebra_homomorphism}
	Let $\alg$ and $\algtwo$ be Banach lattice algebras. Then a map $\rep:\alg\to\algtwo$ is a \emph{Banach lattice algebra homomorphism} if $\rep$ is a Banach algebra homomorphism as well as a lattice homomorphism.
\end{definition}

Banach algebra homomorphisms are supposed to be continuous. However, since Banach lattice algebra homomorphisms are, in particular, positive linear maps between Banach lattices, their continuity is, in fact, already automatic.

\begin{definition}\label{def:banach_lattice_algebra_representation}
Let $\alg$ be a Banach lattice algebra, and let $\vl$ be a Dedekind complete Banach lattice. Suppose that $\rep:\alg\to\regularops{\vl}$ is a Banach lattice algebra homomorphism. Then $\rep$ is a \emph{Banach lattice algebra representation of $\alg$ on $\vl$}.
\end{definition}

Let $\alg$ be a Banach algebra. Then the \emph{left regular representation of $\alg$} is the map $\leftreg:\alg\to\bounded(A)$ that is obtained by setting $\leftreg(a_1)a_2\coloneqq a_1a_2$ for $a_1,a_2\in\alg$. The left regular representation of a complex Banach algebra is similarly defined.

Let $\alg$ be a Dedekind complete Banach lattice algebra. Since $\alg=\pos{\alg}-\nega{\alg}$, it follows that the left regular representation $\leftreg$ of $\alg$ is, in fact, a positive algebra homomorphism $\leftreg:\alg\to\regularops{\alg}\subseteq\bounded(A)$ from $\alg$ into the regular operators on $\alg$. Since $\alg$ is a Dedekind complete Banach lattice, it is a meaningful question whether the left regular representation $\leftreg$ of $\alg$ as a Banach algebra is, in fact, a Banach lattice algebra representation of $\alg$ on itself. That is, is the map $\leftreg:\alg\to\regularops{\alg}$ a lattice homomorphism? This question is raised in \cite[Problem ~1]{wickstead:2017c}. In \cref{rem:left_regular_representation_overview}, below, we summarise what is known to us.

\smallskip

We shall now introduce complex Banach lattice algebras.

Let $\alg$ be a Banach lattice algebra with norm $\norm{\,\cdot\,}$. Applying the general procedure for the complexification of a Banach lattice, one obtains the complex Banach lattice $(\compalg,\compnorm{\,\cdot\,})$. Furthermore, $\compalg$ is also a complex algebra. It is a non-trivial fact that $\compabs{z_1z_2}\leq\compabs{z_1}\compabs{z_2}$ for all $z_1,z_2\in\comp{A}$. We refer to \cite[Lemma~1.5]{arendt_THESIS_TUEBINGEN:1979} or \cite[Satz~1.1]{scheffold:1980} for a proof of this result, which was later generalised to arbitrary Archimedean relatively uniformly complete Riesz algebras in \cite{huijsmans:1985}.  The submultiplicativity of the lattice norm $\norm{\,\cdot\,}$ on $\alg$ then immediately implies that $\compnorm{z_1z_2}\leq\compnorm{z_1}\compnorm{z_2}$ for $z_1,z_2\in\comp{A}$. Hence the complex Banach space $(\compalg, \compnorm{\,\cdot\,})$ is also a complex Banach algebra. The complex Banach space $(\compalg, \compnorm{\,\cdot\,})$, with its structures of a complex Banach lattice and of a complex Banach algebra, is a \emph{complex Banach lattice algebra}.

\begin{example}\label{ex:complex_BLA_of_continuous_functions}
	Let $\ts$ be a topological space. Complexification of the Banach lattice algebra $(\contotsR,\norm{\,\cdot \,}_\infty)$, respectively, $(\contbtsR,\norm{\,\cdot \,}_\infty)$, yields the complex Banach lattice algebra $(\contotsC,\norm{\,\cdot \,}_\infty)$, respectively, $(\contbtsC,\norm{\,\cdot\,}_\infty)$.
\end{example}

\begin{example}\label{ex:complex_BLA_of_operators}
Let $\vl$ be a Dedekind complete Banach lattice. Then $(\regularops{\vl},\regularnorm{\,\cdot\,})$ is a Banach lattice algebra, and complexification yields the complex Banach lattice algebra $(\regularops{\compvl},\norm{\,\cdot\,}_{\tfs{r},\CC})$.
\end{example}

As we shall see later in Section~\ref{sec:locally_compact_groups}, the complex group algebra (respectively, the complex measure algebra) of a locally compact group can be identified, as a complex algebra, with the complexification of the group algebra (respectively, the measure algebra) of the group. It is not difficult to see that the usual norms on these two complex Banach algebras coincide with the norms they obtain as complexifications of the pertinent Banach lattice algebras. Hence the complex group algebra and the complex measure algebra of a locally compact group, with the usual norm, are both complex Banach lattice algebras.

\begin{remark}
It is possible to complexify arbitrary Banach algebras. Indeed, suppose that $\alg$ is a Banach algebra. Then the algebraic complexification $\compalg$ can be given a norm $\compnorm{\,\cdot\,}$ such that $(\compalg,\compnorm{\,\cdot\,})$ is a complex Banach algebra and the natural embedding $a\mapsto(a,0)$ from $\alg$ into $\compalg$ is an isometry. Furthermore, all norms on $\compalg$ with this property are equivalent. We refer to \cite[Theorem~1.3.2]{rickart_GENERAL_THEORY_OF_BANACH_ALGEBRAS:1960} for these results.

There is, in fact, an explicit construction of such a norm in \cite{rickart_GENERAL_THEORY_OF_BANACH_ALGEBRAS:1960}. It would be interesting to investigate whether, for the complexifications of the Banach lattice algebras in the present article, this particular norm in \cite{rickart_GENERAL_THEORY_OF_BANACH_ALGEBRAS:1960} coincides with the norm as found above via the complexification of Banach lattices. If this were even true for general Banach lattice algebras, then this would yield an alternative proof of the submultiplicativity of the norm found via the complexifications of Banach lattices that would not need the results in \cite[Lemma~1.5]{arendt_THESIS_TUEBINGEN:1979},  \cite{huijsmans:1985}, or \cite[Satz~1.1]{scheffold:1980} referred to above.
\end{remark}

\begin{definition}\label{def:complex_banach_lattice_algebra_homomorphism}
Let $\alg$ and $\algtwo$ be Banach lattice algebras. Then a map $\rep:\compalg\to\compalgtwo$ is a \emph{complex Banach lattice algebra homomorphism} if $\rep$ is a complex Banach algebra homomorphism as well as a complex lattice homomorphism.
\end{definition}

Let $\alg$ and $\algtwo$ be Banach lattice algebras. Then a map $\rep:\compalg\to\compalgtwo$ is a complex Banach lattice algebra homomorphism if and only if $\rep$ maps $\alg$ into $\algtwo$ and the restricted map $\rep\mid_\alg:\alg\to\algtwo$ is a Banach lattice algebra homomorphism.

A complex Banach lattice homomorphism is automatically continuous.

\begin{definition}\label{def:complex_banach_lattice_algebra_representation}
Let $\alg$ be a Banach lattice algebra, and let $\vl$ be a Dedekind complete Banach lattice. Suppose that $\rep:\compalg\to\regularops{\compvl}$ is a complex Banach lattice algebra homomorphism. Then $\rep$ is a \emph{complex Banach lattice algebra representation of $\alg$ on $\compvl$}.
\end{definition}

Let $\alg$ be a Banach lattice algebra, and let $\vl$ be a Dedekind complete vector lattice.  Then, by combining \cref{def:banach_lattice_algebra_homomorphism,def:banach_lattice_algebra_representation,def:complex_banach_lattice_algebra_homomorphism,def:complex_banach_lattice_algebra_representation}, we see that a complex algebra homomorphism $\rep:\compalg\to\regularops{\compvl}$ is a complex Banach lattice algebra representation of $\compalg$ on $\compvl$ if and only if $\rep$ maps $\alg$ into $\regularops{\vl}$ and the restricted map $\rep\mid_\alg:\alg\to\regularops{\vl}$ is a Banach lattice algebra representation of $\alg$ on $\vl$.

Let $\alg$ be a Dedekind complete Banach lattice algebra. Then the left regular representation $\leftreg$ of the complex Banach algebra $\compvl$ is a positive algebra homomorphism $\leftreg:\compvl:\to\regularops{\compvl}$. The left regular representation of $\compalg$ is a complex Banach lattice algebra representation of $\compalg$ on itself if and only if the left regular representation of $\alg$ is a Banach lattice algebra representation of $\alg$ on itself.

\smallskip

We mention the following. Let $\alg$ be a complex Banach algebra. Suppose that $^\ast:\alg\to\alg$ is a conjugate-linear map such that $(a^\ast)^\ast=a$ for $a\in\alg$, $(a_1 a_2)^\ast=a_2^\ast a_1^\ast$ for $a_1,a_2\in\alg$, and $\norm{a^\ast}=\norm{a}$ for $a\in\alg$. Then the map $^\ast$ is an \emph{involution on $\alg$}, and $\alg$ is a \emph{complex Banach $^\ast$-algebra}. For complex Banach $^\ast$-algebras, see \cite{palmer_BANACH_ALGEBRAS_AND_THE_GENERAL_THEORY_OF_STAR-ALGEBRAS_VOLUME_I:1994, palmer_BANACH_ALGEBRAS_AND_THE_GENERAL_THEORY_OF_STAR-ALGEBRAS_VOLUME_II:2001}, for example. The theory of $^\ast$-representations of complex Banach $^\ast$-algebras on complex Hilbert spaces is well developed.

In our context, one can consider complex Banach lattice algebras that are also complex Banach $^\ast$-algebras. Examples are $\contotsC$ and $\contbtsC$ for a topological space $\ts$, provided with complex conjugation as involution. The complex group algebra and the complex measure algebra of a locally compact group are other natural examples of complex Banach lattice $^\ast$-algebras. However, there does not seem to be a natural r\^ole for the involution in the representation theory of Banach lattice $^\ast$-algebras. The reason is that the complex Banach lattice algebra $\regularops{\compvl}$, where $\vl$ is a Dedekind complete Banach lattice, does not have a natural involution. It has a natural conjugation, but this preserves the order of the factors in a product of linear operators rather than reverses it.

\section{Locally compact spaces}\label{sec:locally_compact_spaces}

\noindent
In this section, we shall let $\ts$ denote a non-empty, locally compact space. As for all topological spaces in this article, $\ts$ is supposed to be Hausdorff.

We shall be concerned with the order dual $\contctsRod$ of $\contctsR$. As explained in Section~\ref{sec:introduction_and_overview}, the r{\^o}le of $\contctsRod$ in the present article is to be present as a large vector lattice that contains various familiar vector lattices as sublattices; see \cref{res:embedding_of_Lp_into_ordercontcRod,res:embedding_of_M(X)_as_banach_lattice_with_properly_separated_elements}, below, for example.

\smallskip

The first step to be taken is to observe that $\contctsRod$ is equal to the space of real Radon measures on $\ts$ in the sense of Bourbaki \cite{bourbaki_INTEGRATION_VOLUME_I_CHAPTERS_1-6_SPRINGER_EDITION:2004}. This will make a few (not too deep) known results for these Radon measures and their supports available. For this, we shall briefly recall the definition of Bourbaki's Radon measures on $\ts$.

As usual, for a real- or complex-valued function $f$ on $\ts$, the \emph{support of $f$}, denoted by $\supp{f}$, is the closure of the set consisting of those $x\in\ts$ such that $f(x)\neq 0$.

For each non-empty subset $S$ of $\ts$, we let $\contc(\ts,\RR;S)$ denote the set of those $f\in\contctsR$ such that $\supp{f}\subseteq S$. Let $K$ be a non-empty, compact subset of $\ts$. With the uniform norm, $\contc(\ts,\RR;K)$ is a (possibly zero) Banach space. The space $\contctsR$ is the union of the spaces $\contc(\ts,\RR;K)$ as $K$ runs over all non-empty, compact subsets of $\ts$. Consider the family $\mathcal N$ of all absorbing, symmetric, convex subsets $V$ of $\contctsR$ such that $V\cap\contc(\ts,\RR;K)$ is a neighbourhood of $0$ in $\contc(\ts,\RR;K)$ for each non-empty, compact subset $K$ of $\ts$. According to \cite[II, $\mathsection$~4, No.~4, Proposition~5]{bourbaki_TOPOLOGICAL_VECTOR_SPACES_CHAPTERS_1-5_SPRINGER_EDITION:1987}, $\mathcal N$ is a local base at $0$ for a locally convex vector space topology $\mathcal T$ on $\contctsR$. Furthermore, a linear map from $\contctsR$ into a locally convex space is continuous with respect to $\mathcal T$ if and only if its restriction to $\contc(\ts,\RR;K)$ is continuous for each non-empty, compact subset $K$ of $\ts$, and $\mathcal T$ is the only locally convex topology on $\contctsR$ with this property. The topology $\mathcal T$ is also the strongest locally convex topology on $\contctsR$ such that the inclusion map from $\contc(\ts,\RR;K)$ into $\contctsR$ is continuous for each non-empty, compact subset $K$ of $\ts$. The topology $\mathcal T$ on $\contctsR$ is called the \emph{direct limit} or \emph{inductive limit} of the topologies on the spaces $\contc(\ts,\RR;K)$ for non-empty, compact subsets $K$ of $\ts$.

A real \emph{Radon measure on $\ts$} in the sense of Bourbaki is a real-valued linear functional on $\contctsR$ that is continuous with respect to the topology $\mathcal T$ specified above; see \cite[III, $\mathsection$~1, No.~3, Definition~2]{bourbaki_INTEGRATION_VOLUME_I_CHAPTERS_1-6_SPRINGER_EDITION:2004}. In \cite{bourbaki_INTEGRATION_VOLUME_I_CHAPTERS_1-6_SPRINGER_EDITION:2004}, Bourbaki uses the notation $\mathcal{M}(\ts;\mathbf{R})$ for the space of real Radon measures on $\ts$.

An alternative description of $\contctsRod$ is given by the following result. It can already be found in the literature as \cite[paragraph preceding III, $\mathsection$~1, No.~5, Theorem~3]{bourbaki_INTEGRATION_VOLUME_I_CHAPTERS_1-6_SPRINGER_EDITION:2004}, but we thought it worthwhile to make it explicit and also to include the easy proof, as we wish to combine some of the available results on Bourbaki's Radon measures with their lattice structure, which is not as prominent in Bourbaki as we shall need it.

\begin{proposition}\label{res:radon_measures_are_the_order_dual}
	Let $\ts$ be a non-empty, locally compact space. Then $\contctsRod$ is the space  $\mathcal{M}(\ts;\mathbf{R})$ of real Radon measures on $\ts$ in the sense of Bourbaki.
\end{proposition}

\begin{proof}
	Suppose that $\ode:\contctsR\to\RR$ is a Radon measure in the sense of Bourbaki. Let $S\subseteq\contctsR$ be an order bounded subset. Then there exists $g\in\contctsR$ such that $\abs{f}\leq g$ for all $f\in S$. This implies that $S$ is a uniformly bounded subset of $\contc(\ts,\RR;\supp{g})$. Since the restriction of $\ode$ to $\contc(\ts,\RR;\supp{g})$ is continuous, $\ode(A)$ is a bounded, and then also an order bounded, subset of $\RR$. Hence $\ode\in\contctsRod$.
	
	Conversely, suppose that $\ode\in\contctsRod$. Let $K$ be a non-empty, compact subset of $\ts$. Then the restriction of $\ode$ to $\contc(\ts,\RR;K)$ is a regular linear functional. Since $\contc(\ts,\RR;K)$ is a Banach lattice, this restriction is continuous. Hence $\ode$ is a Radon measure in the sense of Bourbaki.
\end{proof}

Let $\ts$ be a non-empty, locally compact space. The above proposition makes it slightly easier to see that a linear functional on $\contctsR$ is a Radon measure. Indeed, it will usually be obvious that it is regular if this be, in fact, the case, whereas seeing that it is continuous on each subspace $\contc(\ts,\RR;K)$ could be (marginally) more complicated.

\smallskip

It is now also possible to make contact with measure theory in the other, perhaps more usual, sense of the word. In order to do so, we recall that a positive measure $\mu:\sigmats\to[0,\infty]$ on the Borel $\sigma$-algebra $\sigmats$ of $\ts$ is:
\begin{enumerate}
	\item a \emph{Borel measure} if  $\mu(K)<\infty$ for all compact subsets $K$ of $\ts$;
	\smallskip
	\item \emph{outer regular on $A\in\sigmats$} if $\mu(A)=\inf\desset{\mu(V): V \textup{ open and }A\subseteq V}$;
	\smallskip
	\item \emph{inner regular on $A\in\sigmats$} if $\mu(A)=\sup\desset{\mu(K): K \textup{ compact and }K\subseteq A}$.	
\end{enumerate}
Using the terminology in \cite[p.~352]{aliprantis_burkinshaw_PRINCIPLES_OF_REAL_ANALYSIS_THIRD_EDITION:1998}, $\mu$ is a \emph{positive regular Borel measure on $\ts$} if it is a positive Borel measure that is outer regular on all $A\in\sigmats$ and inner regular on all open subsets of $\ts$. The measure $\mu$ is \emph{finite} if $\mu(\ts)<\infty$.

The nomenclature is not uniform in the literature; sometimes the inner regularity on all elements of $\sigmats$ rather than just on the open subsets is incorporated in the definition of a regular Borel measure, as in \cite[p.~212]{folland_REAL_ANALYSIS_SECOND_EDITION:1999}. In \cite[p.~212]{folland_REAL_ANALYSIS_SECOND_EDITION:1999}, our positive regular Borel measures are called Radon measures. In view of the possibility of confusion with Bourbaki's terminology, we prefer to speak of positive regular Borel measures in the present article.

We shall now review a number of properties of regular measures on $\ts$. Details can be found in \cite{aliprantis_burkinshaw_PRINCIPLES_OF_REAL_ANALYSIS_THIRD_EDITION:1998}, for example; this reference puts more emphasis on the lattice structure than several other sources.

The set of positive regular Borel measures on $\ts$ is a cone that is denoted by $\regularextposmeasts$. Its subcone consisting of the finite positive regular Borel measures on $\ts$ is denoted by $\regularposmeasts$. By definition, the real-linear span of $\regularposmeasts$ is the vector space $\regularrealmeasts$ of \emph{real regular Borel measures on $\ts$}. The vector space $\regularrealmeasts$ is, in fact, a Dedekind complete Banach sublattice of the Banach lattice $\meas(\ts,\borelsalg,\RR)$ from \cref{ex:spaces_of_measures_two}. The supremum of two elements is given by \cref{eq:supremum_of_two_measures}, the modulus  by \cref{eq:norm_of_a_measure}, and the norm by \cref{eq:norm_of_a_measure}.

Let $\ode\in\contctsRod$. After splitting $\ode$ into its positive and negative parts, the Riesz representation theorem for positive functionals on $\contctsR$ implies that there exist $\pos{\mu}, \nega{\mu}\in\regularextposmeasts$ such that
\begin{equation}\label{eq:riesz_representation_theorem}
\pairing{\ode,f}=\int_\ts\!f\di{\pos{\mu}}-\int_\ts\!f\di{\nega{\mu}}
\end{equation}
for all $f\in\contctsR$. If $\ode\geq 0$, then one can take $\nega{\mu}=0$, and in this case $\pos{\mu}$ is uniquely determined.

Let $\ode\in\contctsRod$, and suppose that $\ode$ is a continuous linear functional on $(\contctsR,\norm{\,\cdot\,}_\infty)$; equivalently, one can suppose that $\ode$ is the restriction to $\contctsR$ of a continuous linear functional on $(\contotsR,\norm{\,\cdot\,}_\infty)$. Then $\pos{\mu}$ and $\nega{\mu}$ in \cref{eq:riesz_representation_theorem} can both be taken to be elements of $\regularposmeasts$. Conversely, if $\pos{\mu},\nega{\mu}\in\regularposmeasts$, then the right-hand side of \cref{eq:riesz_representation_theorem} defines a continuous linear functional $\ode$ on $(\contotsR,\norm{\,\cdot\,}_\infty)$. In this way, an isometric isomorphism of Banach lattices between the norm (or order) dual of the Banach lattice $(\contotsR,\norm{\,\cdot\,}_\infty)$ and the Banach lattice $(\regularrealmeasts,\norm{\,\cdot\,})$ is obtained; see \cite[Theorem~38.7]{aliprantis_burkinshaw_PRINCIPLES_OF_REAL_ANALYSIS_THIRD_EDITION:1998}, for example.

\begin{remark}
The measures $\pos{\mu}$ and $\nega{\mu}$ in \cref{eq:riesz_representation_theorem} can be infinite simultaneously, so that it is meaningless to say that $\ode$ is represented by the measure $\pos{\mu}-\nega{\mu}$ because the latter cannot generally be properly defined. This is where Bourbaki's terminology for Radon `measures' conflicts with that in measure theory in the sense of Lebesgue and Caratheodory.
\end{remark}

\smallskip

Let $\ts$ be a non-empty, locally compact space. The Riesz representation theorem provides a means to define the product of a bounded Borel measurable function on $\ts$ and an element of $\contctsRod$. We shall now explain this.

Let $\ode\in\contctsRod$. Suppose that $U$ is a non-empty, open, and relatively compact subset of $\ts$. Since $\contc(\ts,\RR;U)\subseteq \contc(\ts,\RR;\overline{U})$, the restriction of $\ode$ to $\contc(\ts,\RR;U)$ is continuous when $\contc(\ts,\RR;U)$ is supplied with the uniform norm. Therefore, there exists a unique finite regular Borel measure $\mu$ on $U$ such that
 \[
 \pairing{h,\ode}=\int_U\! h\di{\mu_U}
 \]
 for all $h\in\contc(\ts,\RR;U)$.

 Suppose that $V$ is an open and relatively compact subset of $\ts$ with $V\supseteq U$. Then it is a consequence of  \cite[Section~7.2, Exercise~7]{folland_A_COURSE_IN_ABSTRACT_HARMONIC_ANALYSIS_SECOND_EDITION:2016} and the uniqueness part of the Riesz representation theorem that $\mu_U$ equals the restriction of $\mu_V$ to $U$. Consequently, suppose that $U$ and $V$ are two non-empty, open, and relatively compact subsets of $\ts$ such that $U\cap V\neq\emptyset$. Then the restrictions of $\mu_U$ and $\mu_V$ to $U\cap V$ are identical.

Let $g:\ts\to\RR$ be a bounded Borel measurable function on $\ts$. Suppose that $f\in\contctsR$, and choose an open and relatively compact neighbourhood $U$ of $\supp{f}$ in $\ts$. Since $fg$ is zero outside $U$, it follows from the above that the integral
\begin{equation*}\label{eq:product_definition}
\int_U\!fg\di{\mu_U}
\end{equation*}
does not depend on the choice of $U$. Hence we can set
\[
\pairing{g\ode,f}\coloneqq\int_U\!fg\di{\mu_U}
\]
as a well-defined element of $\RR$, thus obtaining a map $g\ode:\contctsR\to\RR$. It is then routine to verify that $g\ode\in\contctsRod$, and that $g\ode$ depends bilinearly on the bounded Borel measurable function $g$ on $\ts$ and the element $\ode$ of $\contctsRod$. The element $g\ode$ of $\contctsRod$ is the \emph{product of $g$ and $\ode$}.

Although we shall not need this, let us note that, more generally, a similar argument that is based on local applications of the Riesz representation theorem can be employed to define the product $g\ode$ of a Borel measurable function $g$ on $\ts$ that is locally integrable (in the canonical sense) with respect to $\abs{\ode}$ for a given $\ode\in\contctsRod$. It is possible to avoid the Riesz representation theorem in defining such products, see \cite[V, $\mathsection$~5. No.~2]{bourbaki_INTEGRATION_VOLUME_I_CHAPTERS_1-6_SPRINGER_EDITION:2004}, but the definition using the Riesz representation theorem may be a little more transparent.

\smallskip

Following Bourbaki (see \cite[III, $\mathsection$~2, Nos.~1 and ~2]{bourbaki_INTEGRATION_VOLUME_I_CHAPTERS_1-6_SPRINGER_EDITION:2004}), we shall now introduce the supports of elements of $\contctsRod$.

Let $U$ be non-empty, open subset of $\ts$. An element $\ode$ of $\contctsRod$ \emph{vanishes on $U$} if
$\pairing{\ode,f}=0$ for all $f\in\contc(\ts,\RR;U)$. By definition, $\ode$ vanishes on the empty set. A partition of unity argument shows that $\ode$ vanishes on the open subset $\mathcal U$ of $\ts$ that is the union of all open subsets of $\ts$ on which $\ode$ vanishes. The closed subset $\ts\setminus\mathcal U$ of $\ts$ is called the \emph{support of $\ode$}; it is denoted by $\supp{\ode}$. Thus a point $x$ in $\ts$ is in the support of $\ode$ if and only if, for every open neighbourhood $U$ of $x$, there exists $f\in\contc(\ts,\RR;U)$ such that $\pairing{\ode,f}\neq 0$.

Let $\ode\in\contctsRod$. Then $\supp{\ode}=\supp{\abs{\ode}}=\supp{\pos{\ode}\,\cup\,\supp{}\nega{\ode}}$; see \cite[III, $\mathsection$~2, No.~2, Propositions~2]{bourbaki_INTEGRATION_VOLUME_I_CHAPTERS_1-6_SPRINGER_EDITION:2004}.

Let $\ode_1,\ode_2\in\contctsRod$. Then $\supp{(\ode_1+\ode_2)}\subseteq\supp{\ode_1}\cup\supp{\ode_2}$, and if $\abs{\ode_1}\leq\abs{\ode_2}$, then $\supp{\ode_1}\subseteq\supp{\ode_2}$; see \cite[III, $\mathsection$~2, No.~2, Propositions~3 and~4]{bourbaki_INTEGRATION_VOLUME_I_CHAPTERS_1-6_SPRINGER_EDITION:2004}. Consequently, if $S$ is an arbitrary subset of $\ts$, then the subset of $\contctsRod$ consisting of all elements $\ode$ of $\contctsRod$ such that $\supp{\ode}\subseteq S$ is an order ideal of $\contctsRod$.

Let $\ode\in\contctsRod$. It can happen that $\supp{\pos{\ode}}=\supp{\nega{\ode}}=\ts$; see \cite[V, Exercises, $\mathsection$~5, Exerc.~4]{bourbaki_INTEGRATION_VOLUME_I_CHAPTERS_1-6_SPRINGER_EDITION:2004}. Hence the disjointness of two elements of $\contctsRod$ does not imply that their supports are disjoint subsets of $\ts$. The following result shows that the converse implication \emph{does} hold.

\begin{lemma}\label{res:spatial_and_lattice_disjointness}
	Let $\ts$ be a non-empty, locally compact space. Let $\ode_1,\ode_2\in\contctsRod$ be such that $\supp{\ode_1}$ and $\supp{\ode_2}$ are disjoint subsets of $\ts$. Then $\ode_1$ and $\ode_2$ are disjoint elements of $\contctsRod$. Consequently, $\abs{\ode_1+\ode_2}=\abs{\ode_1}+\abs{\ode_2}$.
\end{lemma}

\begin{proof}
	Using the fact that $\supp{\ode}=\supp{\abs{\ode}}$ for $\ode\in\contctsRod$, we may suppose that $\ode_1,\ode_2\in\pos{(\contctsRod)}$.
	
	Then \cref{eq:RK3} yields that, for $f\in\pos{\contctsR}$, we have
	\[
	(\ode_1\wedge\ode_2)(f)=\inf\desset{\ode_1(f_1)+\ode_2(f_2): f_1,f_2\in\pos{\contctsR},\,f_1+f_2=f}.
	\]
	Since $\supp{\ode_1}$ and $\supp{\ode_2}$ are disjoint, we have
	\[
	\supp{f}\subseteq\ts=\left(\ts\setminus\supp{\ode_1}\right)\cup\left(\ts\setminus\supp{\ode_2}\right).
	\]
	We can then find continuous functions $g_1,g_2:\ts\to[0,1]$ with compact support such that $g_1+g_2=1$, $\supp{g_1}\subseteq \ts\setminus\supp{\ode_1}$, and $\supp{g_2}\subseteq \ts\setminus\supp{\ode_2}$. For the resulting decomposition $f=g_1f+g_2f$, we have $\ode_1(g_1f)=\ode_2(g_2f)=0$, and this shows that $(\ode_1\wedge\ode_2)(f)\leq 0$. Since obviously  $(\ode_1\wedge\ode_2)(f)\geq 0$, we see that  $(\ode_1\wedge\ode_2)(f)=0$. Hence $\ode_1\wedge\ode_2=0$.
	
	Now that we have established that $\ode_1$ and $\ode_2$ are disjoint, the final statement follows from the general principle in vector lattices that the modulus is additive on the sum of two (in fact, of finitely many) mutually disjoint elements.
\end{proof}

\begin{remark}

\cref{res:spatial_and_lattice_disjointness}, with its elementary proof, is also a consequence of the technically considerably more demanding
\cite[V, $\mathsection$~5, No.~7, Proposition~13]{bourbaki_INTEGRATION_VOLUME_I_CHAPTERS_1-6_SPRINGER_EDITION:2004}, where a necessary and sufficient condition for two elements of $\contctsRod$ to be disjoint\textemdash Bourbaki calls such elements \emph{alien} (to each other)\textemdash is given. The reader may wish to consult \cite[IV, $\mathsection$~2, No.~2, Proposition~5 and IV, $\mathsection$~5, No.~2, Definition~3]{bourbaki_INTEGRATION_VOLUME_I_CHAPTERS_1-6_SPRINGER_EDITION:2004} to see that an element $\ode$ of $\contctsRod$ is \emph{concentrated on $\supp{\ode}$} in the sense of \cite[V, $\mathsection$~5, No.~7, Definition~4]{bourbaki_INTEGRATION_VOLUME_I_CHAPTERS_1-6_SPRINGER_EDITION:2004}, after which it is immediate from \cite[V, $\mathsection$~5, No.~7, Proposition~13] {bourbaki_INTEGRATION_VOLUME_I_CHAPTERS_1-6_SPRINGER_EDITION:2004} that the disjointness of the supports of two elements of $\contctsRod$ implies their disjointness in the vector lattice $\contctsRod$.
\end{remark}

The relevance of the following result will become clear in the proof of \cref{res:embedding_of_Lp_into_ordercontcRod}, below.

\begin{lemma}\label{res:separating_supports_for_continuous_compactly_supported_functions}
	Let $\ts$ be a non-empty, locally compact space, and let $f\in\contctsR$. Take $\varepsilon>0$. Then there exist $\pos{g},\nega{g}\in\contctsR$ such that:
	\begin{enumerate}
		\item $0\leq \pos{g}\leq \pos{f}$ and $0\leq \nega{g}\leq \nega{f}$;
		\item $0\leq \pos{f}-\pos{g}\leq\varepsilon\indicator_\ts$ and $0\leq \nega{f}-\nega{g}\leq\varepsilon\indicator_\ts$;
		\item  $\supp{\pos{g}}\cap\supp{\nega{g}=\emptyset}$.
	\end{enumerate}
\end{lemma}

\begin{proof}
	If $\pos{f}=0$, then we can take $\pos{g}=0$ and $\nega{g}=\nega{f}$; if $\nega{f}=0$, then we can take $\pos{g}=\pos{f}$ and $\nega{g}=0$. Hence we may suppose that there exists $\delta>0$ such that $\desset{x\in\ts: \pos{f}(x)>\delta}$ and $\desset{x\in\ts: \nega{f}(x)>\delta}$ are both non-empty subsets of $\ts$. It is then sufficient to prove the result for all $\varepsilon$ such that $0<\varepsilon<\delta$. For such a fixed $\varepsilon$, set 
	\[
	\pos{g}\coloneqq (\pos{f}\vee\varepsilon\indicator_\ts)-\varepsilon\indicator_\ts.
	\] 
	Then $\pos{g}\in\conttsR$, $0\leq \pos{g}\leq \pos{f}$, and $0\leq \pos{f}-\pos{g}\leq\varepsilon\indicator_\ts$; we see that $\pos{g}\in\contctsR$. Likewise, we set 
	\[
	\nega{g}\coloneqq (\nega{f}\vee\varepsilon\indicator_\ts)-\varepsilon\indicator_\ts,
	\] 
	and then $\nega{g}\in\contctsR$, $0\leq \nega{g}\leq \nega{f}$, and $0\leq \nega{f}-\nega{g}\leq\varepsilon\indicator_\ts$.
	
	Let $x\in\ts$. If 
	\[
	x\in\supp{\pos{g}}=\overline{\desset{x\in\ts : \pos{g}(x)\neq 0}}\subseteq \overline{\desset{x\in\ts : \pos{f}(x)>\varepsilon}},
	\] 
	then the continuity of $\pos{f}$ implies that $\pos{f}(x)\geq\varepsilon$. Hence $f(x)\geq\varepsilon$. Likewise, if $x\in\supp{\nega{g}}$, then $\nega{f}(x)\geq\varepsilon$, which implies that $f(x)\leq -\varepsilon$. Since $\varepsilon>0$, this shows that $\supp{\pos{g}}\cap\supp{\nega{g}=\emptyset}$.
\end{proof}

\section{Closed subspaces of locally compact spaces}\label{sec:closed_subspaces_of_locally_compact_spaces}

\noindent Let $\ts$ be a non-empty, locally compact space, and let $\tstwo$ be a non-empty, closed subspace of $\ts$. Then $\tstwo$ is again a locally compact space. We shall now prove that $\contctstwoRod$ can be canonically viewed as the order ideal of $\contctsRod$ that consists of those elements of $\contctsRod$ with support contained in $\tstwo$. The reader who is interested in Banach lattices on groups, but not on semigroups, can omit this section in its entirety.

We are not aware of references for the results in this section, which may find applications elsewhere.

\smallskip

Let $\tstwo$ be a non-empty, closed subset of a locally compact space $\ts$. Then we define the \emph{restriction map} $\res:\contctsR\to\contctstwoR$ by setting $\res f\coloneqq f\mid_Y$ for $f\in\contctsR$. As we shall see, the order adjoint
\[
\resdual:\contctstwoRod\to\contctsRod
\]
of $\res$ is injective, and the image of $\contctstwoRod$ under $\resdual$ is the order ideal of $\contctsRod$ that consists of those elements of $\contctsRod$ with support contained in $\tstwo$.

We shall require two preparatory results. The first one is a slight strengthening of a version of Tietze's extension theorem  \cite[Theorem~20.4]{rudin_REAL_AND_COMPLEX_ANALYSIS_THIRD_EDITION:1987}, on which it is also based.

\begin{proposition}\label{res:extension}
	Let $\ts$ be a non-empty, locally compact space, let $\tstwo$ be a non-empty, closed subspace of $\ts$, and let $f\in\contctstwoR$. Then there exists $F\in\contctsR$ such that $\res F=f$ and $\norm{F}_{\infty,\ts}=\norm{f}_{\infty,\tstwo}$. If $f\geq 0$, then it can be arranged that also $F\geq 0$.
\end{proposition}

\begin{proof}
	Let $f\in\contctstwoR$. Take a relatively compact open neighbourhood $U$ of $\supp{f}$ in $\ts$. Since $\overline{U}\cap\tstwo$ is a compact subset of $\overline{U}$, Tietze's extension theorem shows that there exists an element $g$ of $\cont(\overline{U},\RR)$ such that $g\mid \overline{U}\cap\tstwo=f\mid \overline{U}\cap\tstwo$ as well as   $\norm{g}_{\infty,\overline{U}}=\norm{f}_{\infty,\overline{U}\cap\tstwo}=\norm{f}_{\infty,\tstwo}$.
	By a version of Urysohn's lemma \cite[Theorem~2.12]{rudin_REAL_AND_COMPLEX_ANALYSIS_THIRD_EDITION:1987},  there exists $h\in\cont(\overline{U},\RR)$ such that $h(\overline{U})\subseteq[0,1]$, $h(y)=1$ for $y\in\supp{f}$, and $\supp{h} \subseteq U$.
	
	Set $F\coloneqq gh $, so that $F\in\cont(\overline U,\RR)$  and $\supp{F} \subseteq U$.  We extend $F$ to be an element of $\contctsR$ by setting $F(x)\coloneqq 0$ for $x\in \ts \setminus\overline{U}$. Then we have $\norm{F}_{\infty,
	\ts}\leq\norm{g}_{\infty,\overline{U}}=\norm{f}_{\infty,\tstwo}$.
	
	For $y\in\supp{f}$, we have $F(y)= g(y)h(y) = g(y)= f(y)$; this also shows that $\norm{F}_{\infty,\ts}\geq\norm{f}_{\infty,\tstwo}$.  For $y\in (\overline{U}\cap \tstwo)\setminus \supp{f}$, we have $F(y)= 0=f(y)$ because $g(y)=f(y)=0$. For $y\in \tstwo\setminus\overline{U}$, we have $F(y)=0=f(y)$ because $F$ vanishes on $\ts \setminus \overline{U}$. We conclude that $\res F=f$ and that $\norm{F}_{\infty,\ts}=\norm{f}_{\infty,\tstwo}$.
	
	If $f\geq 0$, then replacing $F$ by $\pos{F}$ shows that we can also arrange that $F\geq 0$.
\end{proof}

\begin{corollary}\label{res:restriction}
	Let $\ts$ be a non-empty, locally compact space, and let $\tstwo$ be a non-empty, closed subspace of $\ts$. Then $\res:\contctsR\to\contctstwoR$ is a continuous, interval preserving, and surjective lattice homomorphism.
\end{corollary}

\begin{proof}
	The map $\res$ is clearly a lattice homomorphism, and it is immediate from the properties of the topologies of $\contctsR$ and $\contctstwoR$ that $\res$ is continuous. The surjectivity follows from \cref{res:extension}.
	
	It remains to show that the positive linear operator $\res: \contctsR \to \contctstwoR $  is interval preserving. For this, take $F\in\pos{\contctsR}$, and suppose that $g\in \contctstwoR$ is such that $0\leq g\leq \res F$.
	By \cref{res:extension}, there exists $G\in\pos{\contctsR}$ such that $\res G=g$. Then $0\leq F\wedge G\leq F$ and $\res(F\wedge G)=\res F\wedge \res G=\res F\wedge g=g$. Thus $\res([0,F])=[0,\res F]$, as required.
\end{proof}

\begin{theorem}\label{res:embedding_of_order_duals}
	Let $\ts$ be a non-empty, locally compact space, and let $\tstwo$ be a non-empty, closed subspace of $\ts$. Then $\resdual:\contctstwoRod\to\contctsRod$ is a weak$^\ast$-continuous, injective, and interval preserving lattice homomorphism.
	
	Furthermore, $\supp{\ode}=\supp{\resdual\ode}$ for all $\ode\in\contctstwoRod$.
	
	The image of $\contctstwoRod$ under $\resdual$ is the order ideal of $\contctsRod$ that consists of all elements $\Ode$ of $\contctsRod$ such that $\supp{\Ode}\subseteq\tstwo$.
	
	Suppose that $g$ is a bounded Borel measurable function on $\tstwo$. Extend $g$ to a Borel measurable function $\widetilde g$ on $\ts$ by setting $\widetilde g(x)\coloneqq 0$ for $x\in\ts\setminus\tstwo$. Then $\resdual(g\ode)=\widetilde g\resdual \ode$ for all $\ode\in\contctstwoRod$.
\end{theorem}

\begin{proof}
	In view of \cref{res:restriction,res:dual_is_lattice_homomorphism,res:dual_is_interval_preserving}, it is clear that $\resdual$, which is obviously weak$^\ast$-continuous, is an injective and interval preserving lattice homomorphism.
	
	We turn to the second statement.
	
	Let $\ode\in\contctstwoRod$. Let $x\in\ts$, and suppose that $x\notin\supp{\ode}$. Since $\tstwo$ is a closed subset of $\ts$, $\supp{\ode}$ is a closed subset of $\ts$. Hence there exists an open neighbourhood $U$ of $x$ in $\ts$ such that $U\cap \supp{\ode}=\emptyset$. Let $f\in\contctsR$ be such that $\supp{f}\subseteq U$. If $\res\! f=0$, then certainly $\pairing{\resdual\ode,f}=\pairing{\ode,\res\! f}=0$. If $\res\! f\neq 0 $, then $\res\! f$ is an element of $\contctstwoR$ such that $\supp{\res\! f}\subseteq U\cap \tstwo$. Since $U\cap \tstwo$ is then a non-empty, open subset of $\tstwo$ that is disjoint from $\supp{\ode}$, we have $\pairing{\ode,\res\! f}=0$. Hence $\pairing{\resdual \ode,f}=0$. We conclude that $\resdual \ode$ vanishes on $U$, and hence $x\notin\supp{\resdual\ode}$. It follows that $\supp{\ode}\supseteq\supp{\resdual \ode}$.
	
	For the reverse inclusion, take $x\in\tstwo$, and suppose that $x\in\supp{\ode}$. Let $U$ be an open neighbourhood of $x$ in $\ts$. Take $f\in\contctstwoR$ such that $\supp{f}\subseteq U\cap\tstwo$ and $\pairing{\ode,f}\neq 0$. By \cref{res:extension}, there exists $F\in\contctsR$ such that $\res F=f$, and Urysohn's lemma furnishes $G\in\contctsR$ such that $G=1$ on $\supp{f}$ and $\supp{G}\subseteq U$. Set $H\coloneqq FG$. Then $H\in\contctsR$, $\supp{H}\subseteq U$, and $\res H=f$. We then conclude from  $\pairing{\resdual\ode,H}=\pairing{\ode,\res H}=\pairing{\ode,f}\neq 0$ that $\resdual\ode$ does not vanish on $U$. Hence $x\in\supp\resdual\ode$. This shows that $\supp{\ode}\subseteq\supp{\resdual \ode}$.
	
	We turn to the statement on the range of $\resdual$.
	
	From what we have already established, it is clear that the support of $\resdual\ode$ is contained in $\tstwo$ for all $\ode\in\contctstwoRod$. Conversely, suppose that $\Ode\in\contctsRod$ is such that $\supp{\Ode}\subseteq\tstwo$. We shall establish the existence of a $\ode\in\contctstwoRod$ such that $\resdual\ode=\Ode$, as follows. Let $f\in\contctstwoR$. Using \cref{res:extension}, we choose $F\in\contctsR$ such that $\res F=f$, and we define $\ode:\contctstwoR\to\RR$ by setting  $\pairing{\ode,f}\coloneqq\pairing{\Ode,F}$. We shall show that this is well defined. For this, it is clearly sufficient to show that $\pairing{\Ode,F}=0$ whenever $F\in\contctsR$ is such that $\res F=0$. Fix such an $F$, and choose an open and relatively compact neighbourhood $U$ of $\supp{F}$ in $\ts$. Then there exists a constant $M\geq 0$ such that $\abs{\pairing{\Ode,G}}\leq M\norm{G}_{\infty,\ts}$ for all $G\in\contc(\ts,\RR;\overline{U})$. Let  $\varepsilon>0$ be fixed, and set $V_\varepsilon\coloneqq\desset{x\in\ts: \abs{F(x)}<\varepsilon}$. Since $\res F=0$, $V_\varepsilon$ is an open neighbourhood of $\tstwo$ in $\ts$; in particular, $V_\varepsilon$ is an open neighbourhood of $\tstwo\cap\supp{F}$ in $\ts$. Take an open and relatively compact subset $W_\varepsilon$ of $\ts$ such that $\tstwo\cap \supp{F}\subseteq W_\varepsilon\subseteq\overline{W_\varepsilon}\subseteq V_\varepsilon$, and take $G_\varepsilon\in\contctsR$ such that $0\leq G_\varepsilon\leq 1$, $G_\varepsilon=1$ on $\overline{W_\varepsilon}$, and $\supp{G_\varepsilon}\subseteq V_\varepsilon$.
	
	Let $x\in\ts$, and suppose that $(FG_\varepsilon-F)(x)\neq 0$. Then certainly $G_\varepsilon(x)\neq 1$, so that $x\notin\overline{W_\varepsilon}$. In particular, $x\notin W_\varepsilon$. We conclude that $\supp(FG_\varepsilon-F)\subseteq \ts\setminus W_\varepsilon$. Evidently, $\supp(FG_\varepsilon-F)\subseteq\supp{F}$, so   $\supp(FG_\varepsilon-F)\subseteq(\ts\setminus W_\varepsilon)\cap\supp{F}$. Hence
	\begin{align*}
	\supp(FG_\varepsilon-F)\cap\supp{\Ode}&\subseteq
	\supp(FG_\varepsilon-F)\cap\tstwo
	\\ &\subseteq (\ts\setminus W_\varepsilon)\cap\tstwo\cap\supp{F}=\emptyset,
	\end{align*}
	since $\tstwo\cap\supp{F}\subseteq W_\varepsilon$.
	It follows from this that $\pairing{\Ode,F}=\pairing{\Ode,FG_\varepsilon}$. Since, in addition,  $FG_\varepsilon\in\contc(\ts,\RR;\overline{U})$ and $\norm{FG_\varepsilon}_{\infty,\ts}\leq\varepsilon$, we have $\abs{\pairing{\Ode,FG_\varepsilon}}\leq\varepsilon M_{\overline{U}}$.
	
	We thus see that $\abs{\pairing{\Ode,F}}\leq\varepsilon M_{\overline{U}}$ for all $\varepsilon>0$. Hence $\pairing{\Ode,F}=0$. This establishes our claim.
	
	Now that we know that the map $\ode:\contctstwoR\to\RR$ is well defined, it is immediate that it is linear. Combining the facts that a positive $f\in\contctstwoR$ has a positive extension, as asserted by \cref{res:extension}, and that $\Ode=\pos{\Ode}-\nega{\Ode}$ in $\contctsRod$, it is easy to see that $\ode\in\contctstwoRod$. Finally, for $F\in\contctstwoR$, we have, using the definition of $\ode$, that  $\pairing{\resdual\ode,F}=\pairing{\ode,\res  F}=\pairing{\Ode,F}$. Hence $\resdual\ode=\Ode$.
	
	We have now shown that the image of $\contctstwoRod$ under $\resdual$ is the subset of $\contctsRod$ that consists of all elements $\Ode$ of $\contctsRod$ such that $\supp{\Ode}\subseteq\tstwo$. Since such a subset of $\contctsRod$ is an order ideal of $\contctsRod$ for an arbitrary subset $\tstwo$ of $\ts$, the proof of the statement on the range of $\resdual$ is complete.
	
	We turn to the final statement.
	
	Let $g$ be a bounded Borel measurable function on $\tstwo$, and let $\ode\in\contctstwoRod$. Suppose that $f\in\contctsR$. Choose a non-empty, open, relatively compact neighbourhood $U$ of $\supp{f}$ in $\ts$; we may suppose that $U\cap\tstwo\neq\emptyset$. Then $U\cap\tstwo$ is a non-empty, open, relatively compact neighbourhood of $\supp(\res f)$ in $\tstwo$. There exists a unique regular Borel measure $\mu$ on $U\cap\tstwo$ such that
	\begin{equation}\label{eq:equal_integrals_zero}
	\pairing{\ode,h}=\int_{U\cap Y}\! h\di{\mu}
	\end{equation}
	for all $h\in\contc(\tstwo,\RR;U\cap Y)$. Suppose that $A$ is an arbitrary Borel subset of $U$, and set $\widetilde \mu(A)\coloneqq\mu(A\cap(U\cap Y))$. The fact that $U\cap Y$ is closed in $U$ implies that this defines a regular Borel measure $\widetilde{\mu}$ on $U$. It is easily seen that
	\begin{equation}\label{eq:equal_integrals_one}
	\int_U\! k\di{\widetilde\mu}=\int_{U\cap Y} R_{U\cap Y}k\di{\mu}
	\end{equation}
	for all bounded Borel measurable functions $k$ on $U$.
	
	On the other hand, there exists a unique regular Borel measure $\nu$ on $U$ such that
	\begin{equation}\label{eq:equal_integrals_two}
	\pairing{\resdual\ode,k}=\int_U\! k\di{\nu}
	\end{equation}
for all $k\in\contc(\ts,\RR;U)$.
	
	Combining \cref{eq:equal_integrals_zero,eq:equal_integrals_one,eq:equal_integrals_two}, we see that, for $k\in\contc(\ts,\RR;U)$, we have
	\[
	\int_U\!k\di{\nu}=\pairing{\resdual\ode,k}=\pairing{\ode,\res k}=\int_{U\cap \tstwo}\!R_{U\cap\tstwo}k\di{\mu}=
	\int_U\!k\di{\widetilde{\mu}}.
	\]
	It follows that $\nu=\widetilde{\mu}$.

	Using the definitions of $\widetilde g\resdual\ode$ and $\resdual(g\ode)$, we see that this implies that
	\begin{equation*}
	\pairing{\widetilde g\resdual\ode,f}=\int_U\!\widetilde{g}f\di{\nu}=\int_{U}\!\widetilde{g}f\di{\widetilde{\mu}}=\int_{U\cap Y}\! g\res f\di{\mu}=\pairing{g\ode,\res f}=\pairing{\resdual(g\ode),f}.
	\end{equation*}
	Hence $\widetilde g\resdual\ode=\resdual(g\ode)$.
\end{proof}

We are not aware of earlier results in the vein of \cref{res:embedding_of_order_duals}. Bourbaki introduces  restrictions of his Radon measures in \cite[III, $\mathsection$~2, No.~1 and IV, $\mathsection$~5, No.~7]{bourbaki_INTEGRATION_VOLUME_I_CHAPTERS_1-6_SPRINGER_EDITION:2004}, but does not seem to consider what are essentially extensions as in \cref{res:embedding_of_order_duals}.

\section{Embedding familiar vector lattices into $\contctsRod$}\label{sec:embedding_familiar_vector_lattices}

\noindent In this section, $\ts$ is a non-empty, locally compact space. We shall see how various familiar vector lattices can be embedded into $\contctsRod$.

\smallskip

Let $\mu\in\regularextposmeasts$ be a positive regular Borel measure on $\ts$. Suppose that $g:\ts\to\RR$ is Borel measurable. Then $g$ is \emph{locally integrable with respect to $\mu$}, or \emph{locally $\mu$-integrable} if
\[
\int_K\! \abs{g(x)}\di{\mu}<\infty
\]
for every compact subset $K$ of $\ts$. We shall identify two locally $\mu$-integrable functions $g_1$ and $g_2$ that are locally $\mu$-almost everywhere equal, i.e., which are such that
\[
\mu(\desset{x\in K: g_1(x)\neq g_2(x)})=0
\]
for all compact subsets $K$ of $\ts$.
The equivalence classes of locally $\mu$-integrable functions on $\ts$ form a vector lattice when the vector space operations and ordering are defined pointwise locally almost everywhere using representatives of equivalence classes. The vector lattice of equivalence classes thus obtained is denoted by $\Elloneloc(\ts,\borelsalg,\mu,\RR)$.

We shall shortly show that there exists a canonical lattice isomorphism $\Emb$ from  $\Elloneloc(\ts,\borelsalg,\mu,\RR)$ into $\contctsRod$; see \cref{res:homomorphism_from_L1loc_into_contctsRod}, below. The spaces $\Ell^p(\ts,\borelsalg,\mu,\RR)$ for $1\leq p<\infty$ are sublattices of $\Elloneloc(\ts,\borelsalg,\mu,\RR)$; see \cref{res:sublattices_of_Elloneloc}, below. For $1\leq p<\infty$, the restrictions of $\Emb$ to these sublattices will, therefore, yield embeddings of the vector lattices $\Ell^p(\ts,\borelsalg,\mu,\RR)$ as vector sublattices of $\contctsRod$; see \cref{res:embedding_of_Lp_into_ordercontcRod}, below.

We shall need the following auxiliary result, which can be found as \cite[Corollary to Lusin's Theorem~2.24]{rudin_REAL_AND_COMPLEX_ANALYSIS_THIRD_EDITION:1987}, for example.

\begin{proposition}\label{res:corollary_to_lusins_theorem}
	Let $\ts\!$ be a non-empty, locally compact space, let $\aux$ be a bounded Borel measurable function on $\ts$, let $\mu\in\regularextposmeasts$, and let $A\in\borelsalg$ be such that $\mu(A)<\infty$. Suppose that $\aux$ vanishes outside $A$ and that $\norm{\aux}_\infty\leq 1$. Then there exists a sequence $(\auxn)$ in $\contctsR$ such that $\norm{\auxn}_\infty\leq 1$ for all $n\geq 1$, and $\aux(x)=\lim_{n\to\infty} \auxn(x)$ for $\mu$-almost all $x$ in $\ts$.
\end{proposition}

\begin{proposition}\label{res:homomorphism_from_L1loc_into_contctsRod}
	Let $\ts$ be a non-empty, locally compact space, and suppose that $\mu\in\regularextposmeasts$. For $g\in\Elloneloc(\ts,\borelsalg,\mu,\RR)$, set
	\[
	\pairing{\emb{g},f}\coloneqq\int_\ts\! fg\di{\mu}
	\]
	for $f\in\contctsR$. Then $\emb{g}\in\contctsRod$, and the map $\Emb:g\mapsto\emb{g}$ defines an injective lattice homomorphism $\Emb:\Elloneloc(\ts,\borelsalg,\mu,\RR)\to\contctsRod$. Suppose that $h$ is a bounded Borel measurable function on $\ts$. Then $\emb{hg}=h\emb{g}$. Furthermore, $\supp{\emb{g}}\subseteq\supp{g}$ for $g\in\contctsR$.
\end{proposition}

\begin{proof}
	Let $g\in\Elloneloc(\ts,\borelsalg,\mu,\RR)$. It is clear that $\emb{g}\in\contctsRod$. We shall first prove that $\Emb$ is a lattice homomorphism by showing that $\abs{\emb{g}}=\emb{\abs{g}}$. For this, we apply \cref{eq:RK3} to see that
	\begin{equation}\label{eq:riesz_kantorovich_for_L1loc}
	\begin{split}
	\pairing{\abs{\emb{g}},f}&=\sup\desset{\pairing{\emb{g},h}:h\in\contctsR, \,\abs{h}\leq f}\\
	&=\sup\lrdesset{\int_\ts\! hg\di{\mu}:h\in\contctsR,\,\abs{h}\leq f}
	\end{split}
	\end{equation}
	for $f\in\pos{\contctsR}$.
	
	Fix $f\in\pos{\contctsR}$, and take $h\in\contctsR$ with $\abs{h}\leq f$. Then
	\[
	\int_\ts\! hg\di{\mu}\leq \lrabs{\int_\ts\! hg\di{\mu}}\leq\int_\ts\! \abs{h}\abs{g}\di{\mu}\leq\int_\ts\! f\abs{g}\di{\mu}=\pairing{\emb{\abs{g}},f}.
	\]
	This shows that
	\begin{equation}\label{eq:first_inequality_for_L1loc}
	\sup\lrdesset{\int_\ts\! hg\di{\mu}:h\in\contctsR,\,\abs{h}\leq f}\leq\pairing{\emb{\abs{g}},f}.
	\end{equation}
	
	For the reverse inequality, we define $\aux:\ts\to\RR$ by
	\[
	\aux(x)=
	\begin{cases}
	0&\text{ if }x\notin\supp{f},\\
	\sgn(g)&\text{ if }x\in\supp{f}.
	\end{cases}
	\]
	Since $\supp{f}$ is compact, it has finite $\mu$-measure, so that \cref{res:corollary_to_lusins_theorem} yields a sequence $(\auxn)$ in $\contctsR$ such that $\norm{\auxn}_\infty\leq 1$ for all $n\geq 1$, and $\auxn(x)\to \aux(x)$ for $\mu$-almost all $x$ in $\ts$. Note that $\auxn f\in\contctsR$, that $\abs{\auxn f}\leq f$ for all $n\geq 1$, and that
	\begin{align*}
	\lim_{n\to\infty}\int_\ts\!(\auxn f)g\di{\mu}&=\lim_{n\to\infty}\int_\ts\!\ind_{\supp{f}}f\auxn g\di{\mu}=\int_\ts\!\ind_{\supp{f}}f\abs{g}\di{\mu}\\
	&=\int_\ts\! f\abs{g}\di{\mu}
	=\pairing{f,\emb{\abs{g}}}.
	\end{align*}
	Here the dominated convergence theorem was applied in the second step, and this is valid since $\ind_{\supp{f}}f\abs{g}$ is integrable.  We thus see that
	\begin{equation}\label{eq:second_inequality_for_L1loc}
	\sup\lrdesset{\int_\ts\! hg\di{\mu}:h\in\contctsR,\,\abs{h}\leq f}\geq\pairing{\emb{\abs{g}},f}.
	\end{equation}
	Combining \cref{eq:first_inequality_for_L1loc,eq:second_inequality_for_L1loc}, we obtain
	\[
	\sup\lrdesset{\int_\ts\! hg\di{\mu}:h\in\contctsR,\,\abs{h}\leq f}=\pairing{\emb{\abs{g}},f},
	\]
	and then \cref{eq:riesz_kantorovich_for_L1loc} shows that $\pairing{\abs{\emb{g}},f}=\pairing{\emb{\abs{g}},f}$. Hence  $\abs{\emb{g}}=\emb{\abs{g}}$.
	
	It is now easy to prove that $\Emb$ is injective. Indeed, let $g\in\Elloneloc(\ts,\borelsalg,\mu,\RR)$ be such that $\emb{g}=0$. Then also $\emb{\abs{g}}=\abs{\emb{g}}=0$. Suppose that $K$ is a compact subset of $\ts$, and take $f\in\pos{\contctsR}$ such that $f=1$ on $K$. Then
	\[
	\int_K\! \abs{g}\di{\mu}\leq\int_\ts\!f\abs{g}\di{\mu}=\pairing{\emb{\abs{g}},f}=0.
	\]
	Hence $g$ is locally $\mu$-almost everywhere equal to zero, as required.

	The statements on the multiplication by bounded Borel measurable functions and on supports are clear.
\end{proof}

\begin{remark}
	\cref{res:homomorphism_from_L1loc_into_contctsRod} also  follows from \cite[V, $\mathsection$~5, No.~2, Corollary to Proposition~2]{bourbaki_INTEGRATION_VOLUME_I_CHAPTERS_1-6_SPRINGER_EDITION:2004}. Bourbaki's approach is different from ours. It does not use the dominated convergence theorem, for example, as there are no integrals present at all.
\end{remark}

\begin{lemma}\label{res:sublattices_of_Elloneloc}
Let $\ts$ be a non-empty, locally compact space, let $1\leq p<\infty$, and let $\mu\in\regularextposmeasts$. Then $\Ell^p(\ts,\borelsalg,\mu,\RR)$ is a vector sublattice of $\Elloneloc(\ts,\borelsalg,\mu)$.
\end{lemma}

\begin{proof}
If a measurable function is $\mu$-almost everywhere equal zero, then it is clearly locally $\mu$-almost everywhere equal to zero. Furthermore, H\"older's inequality implies that every $p$-integrable measurable function is locally integrable.  Hence there exists a canonical lattice homomorphism from $\Ell^p(\ts,\borelsalg,\mu,\RR)$ into $\Elloneloc(\ts,\borelsalg,\mu)$. We need to show that this homomorphism is injective. To this end, suppose that $g$ is a measurable function on $\ts$ such that
\[
\int_\ts\!\abs{g}^p\di{\mu}<\infty
\]
and
\[
\int_K\! \abs{g}\di{\mu}=0
\]
for every compact subset $K$ of $\ts$.
For $n=1,2,\dotsc$, set
\[
A_n\coloneqq\desset{x\in\ts: \abs{g(x)}\geq 1/n}.
\]
Then $\mu(A_n)<\infty$ and
\begin{equation}\label{eq:union}
\desset{x\in\ts: g(x)\neq 0}=\bigcup_{n=1}^\infty A_n.
\end{equation}
Take $n\geq 1$. Since $\mu(A_n)<\infty$, \cite[Proposition~7.5]{folland_REAL_ANALYSIS_SECOND_EDITION:1999} shows that
\begin{equation}\label{eq:measure_is_zero}
\mu(A_n)=\sup\desset{\mu(K):K\text{ compact and }K\subseteq A_n}.
\end{equation}
Suppose that $K$ is a compact subset of $\ts$ such that $K\subseteq A_n$. Then
\[
\frac{1}{n}\mu(K)=\int_K\!\frac{1}{n}\di\mu\leq\int_K\!\abs{g}\di{\mu}=0.
\]
Hence \cref{eq:measure_is_zero} shows that $\mu(A_n)=0$, and then \cref{eq:union} implies that $g$ is $\mu$-almost everywhere equal to zero.
\end{proof}

We can now establish our embedding theorem for $\Ell^p$-spaces.

\begin{theorem}\label{res:embedding_of_Lp_into_ordercontcRod}
	Let $\ts$ be a non-empty, locally compact space, let $\mu\in\regularextposmeasts$, and let $1\leq p<\infty$.
	For $g\in\Ell^p(\ts,\borelsalg,\mu,\RR)$, set
	\[
	\pairing{\emb{g},f}\coloneqq\int_\ts\! fg\di{\mu}
	\]
	for $f\in\contctsR$. Then $\emb{g}\in\contctsRod$, and the map $\Emb:g\mapsto\emb{g}$ defines an injective lattice homomorphism  $\Emb:\Ell^p(\ts,\borelsalg,\mu,\RR)\to \contctsRod$. Suppose that $h$ is a bounded Borel measurable function on $\ts$. Then $\emb{hg}=h\emb{g}$.
	
	For $g\in\Ell^p(\ts,\borelsalg,\mu,\RR)$, set $\norm{\emb{g}}\coloneqq \norm{g}_p$, thus making $\Emb(\Ell^p(\ts,\borelsalg,\mu,\RR))$ into a Dedekind complete Banach lattice. Then the set
	\[
	\lrdesset{\emb{g} : g\in\contctsR,\,\supp{\pos{g}}\cap\supp{\nega{g}}=\emptyset}
	\]
	is a dense subset of the Banach lattice $\Emb(\Ell^p(\ts,\borelsalg,\mu,\RR))$. Consequently, the set
	\[
	\lrdesset{\ode\in\Emb(\Ell^p(\ts,\borelsalg,\mu,\RR)) : \supp\ode\text{ is compact and }\supp{\pos{\ode}}\cap\supp{\nega{\ode}}=\emptyset}
	\]
	is a dense subset of the Banach lattice $\Emb(\Ell^p(\ts,\borelsalg,\mu,\RR))$.
\end{theorem}

\begin{proof}
	It is clear from  \cref{res:homomorphism_from_L1loc_into_contctsRod,res:sublattices_of_Elloneloc} that $\Emb$ is an injective lattice homomorphism that is compatible with multiplication by bounded Borel measurable functions. We establish the remaining statements.
	
	It is obvious that $\Ell^p(\ts,\borelsalg,\mu,\RR)$ is a Dedekind complete Banach lattice when the norm is transported via the lattice isomorphism $\Emb$.
	
	We turn to the density statements. Let $h\in\contctsR$, and let $\varepsilon>0$. It follows from \cref{res:separating_supports_for_continuous_compactly_supported_functions} that there exists $g\in\contctsR$ such that $\supp{g}\subseteq\supp{h}$,  $\norm{h-g}_\infty<\varepsilon$, and $\supp{\pos{g}}\cap\supp{\nega{g}}=\emptyset$. Since then $\norm{h-g}_p\leq\varepsilon\mu(\supp h)^{1/p}$ and since $\contctsR$ is dense in $\Ell^p(\ts,\borelsalg,\mu,\RR)$, it follows that
	\[
	\lrdesset{g\in\contctsR: \,\supp{\pos{g}}\cap\supp{\nega{g}}=\emptyset}
	\]
	is a dense subset of $\Ell^p(\ts,\borelsalg,\mu,\RR)$. Applying the isometry $\Emb$, we see that
	\[
	\lrdesset{\emb{g} : g\in\contctsR,\,\supp{\pos{g}}\cap\supp{\nega{g}}=\emptyset}
	\]
	is a dense subset of the Banach lattice $\Emb(\Ell^p(\ts,\borelsalg,\mu,\RR))$.
	
	Suppose that $g\in\contctsR$ is such that $\supp{\pos{g}}\cap\supp{\nega{g}}=\emptyset$.
	It follows from the inclusions $\supp{\emb{\pos{g}}}\subseteq\supp{\pos{g}}$ and $\supp{\emb{\nega{g}}}\subseteq\supp{\nega{g}}$ that we also have $\supp{\emb{\pos{g}}}\cap\supp{\emb{\nega{g}}}=\emptyset$. Hence \cref{res:spatial_and_lattice_disjointness} shows that $\emb{\pos{g}}$ and $\emb{\nega{g}}$ are disjoint elements of $\contctsRod$, and this implies that the equality $\emb{g}=\emb{\pos{g}}-\emb{\nega{g}}$ gives the decomposition of $\emb{g}$ in $\contctsRod$ into its positive and negative part $\pos{\emb{g}}$ and $\nega{\emb{g}}$, respectively. The final density statement is now clear.
\end{proof}

We shall now show that $\regularrealmeasts$ can also be embedded as a vector sublattice of $\contctsRod$. For this, we shall use the following auxiliary result.  It is a slightly rephrased version of \cite[Theorem~6.12]{rudin_REAL_AND_COMPLEX_ANALYSIS_THIRD_EDITION:1987}, which is a consequence of the Radon--Nikod{\'y}m theorem.

\begin{proposition}\label{res:measure_and_its_variation}
	Let $\mu$ be a finite, real-valued measure on a $\sigma$-algebra of subsets of a set $\ts$.  Then there is a measurable function $\aux$ on $\ts$ such that $\abs{\aux(x)}=1$ for all $x\in\ts$ and $\aux\di\mu=\di{\abs{\mu}}$.
\end{proposition}

\begin{proposition}\label{res:embedding_of_M(X)}
	Let $\ts\!$ be a non-empty, locally compact space. For a finite, real-valued measure $\mu\in\regularrealmeasts$, set
	\[
	\pairing{\emb{\mu},f}\coloneqq\int_\ts\! f\di{\mu}
	\]
	for $f\in\contctsR$. Then $\emb{\mu}\in\contctsRod$, and the map $\Emb:\mu\mapsto\emb{\mu}$ defines an injective lattice homomorphism $\Emb:\regularrealmeasts\to\contctsRod$.
\end{proposition}

\begin{proof}
	Let $\mu\in\regularrealmeasts$. We shall prove that $\emb{\abs{\mu}}=\abs{\emb{\mu}}$.
	The proof for this is quite similar to the proof of \cref{res:homomorphism_from_L1loc_into_contctsRod}. Again we apply \cref{eq:RK3} to see that
	\begin{equation}\label{eq:riesz_kantorovich_for_M(X)}
	\begin{split}
	\pairing{\abs{\emb{\mu}},f}&=\sup\desset{\pairing{\emb{\mu},h}:h\in\contctsR, \,\abs{h}\leq f}\\
	&=\sup\lrdesset{\int_\ts\! h\di{\mu}:h\in\contctsR,\,\abs{h}\leq f}
	\end{split}
	\end{equation}
	for $f\in\pos{\contctsR}$.
	
	Fix $f\in\pos{\contctsR}$. If $h\in\contctsR$ and $\abs{h}\leq f$, then
	\[
	\int_\ts\! h\di{\mu}\leq \lrabs{\int_\ts\! h\di{\mu}}\leq\int_\ts\! \abs{h}\di{\abs{\mu}}\leq\int_\ts\! f\di{\abs{\mu}}=\pairing{\emb{\abs{\mu}},f}.
	\]
	This shows that
	\begin{equation}\label{eq:first_inequality_for_M(X)}
	\sup\lrdesset{\int_\ts\! h\di{\mu}:h\in\contctsR,\,\abs{h}\leq f}\leq\pairing{\emb{\abs{\mu}},f}.
	\end{equation}
	
For the reverse inequality, we use the unimodular measurable function $\aux$ such that $\aux \di\mu=\di\abs{\mu}$ that is supplied by \cref{res:measure_and_its_variation}. Since $\abs{\mu}$ is a finite measure, \cref{res:corollary_to_lusins_theorem} yields a sequence $(\auxn)$ in $\contctsR$ such that $\norm{\auxn}_\infty\leq 1$ for all $n\geq 1$, and $\auxn(x)\to\aux(x)$ for $\abs{\mu}$-almost all $x$ in $\ts$.
	Note that $\auxn f\in\contctsR$ and $\abs{\auxn f}\leq f$ for all $n\geq 1$, and that, by the dominated convergence theorem,
	\begin{equation*}
	\lim_{n\to\infty}\int_\ts\! (\auxn f)\di{\mu}=\int_\ts\! f\aux\di{\mu}
	=\int_\ts\! f\di{\abs{\mu}}
	=\pairing{f,\emb{\abs{\mu}}}.
	\end{equation*}
	We thus see that
	\begin{equation}\label{eq:second_inequality_for_M(X)}
	\sup\lrdesset{\int_\ts\! h\di{\mu}:h\in\contctsR,\,\abs{h}\leq f}\geq\pairing{\emb{\abs{\mu}},f}.
	\end{equation}
	Combining \cref{eq:first_inequality_for_M(X),eq:second_inequality_for_M(X)}, we obtain that
	\[
	\sup\lrdesset{\int_\ts\! h\di{\mu}:h\in\contctsR,\,\abs{h}\leq f}=\pairing{\emb{\abs{\mu}},f},
	\]
	and then \cref{eq:riesz_kantorovich_for_M(X)} shows that $\pairing{\abs{\emb{\mu}},f}=\pairing{\emb{\abs{\mu}},f}$. Hence $\abs{\emb{\mu}}=\emb{\abs{\mu}}$.
	
	It follows that $\Emb$ is a lattice homomorphism.
	
	Suppose that $\emb{\mu}=0$. We need to show that $\mu=0$. Since also $\emb{\abs{\mu}}=\abs{\emb{\mu}}=0$, we may suppose that $\mu\geq 0$. Let $V$ be a non-empty, open subset of $V$. One of the explicit formulas in the Riesz representation theorem (see \cite[Theorem~7.2]{folland_REAL_ANALYSIS_SECOND_EDITION:1999}) shows that
	\[
	\mu(V)=\sup\lrdesset{\int_\ts\! f\di{\mu}: f\in\contctsR,\,\supp{f}\subseteq V,\,0\leq f\leq \indicator_\ts}.
	\]
	Since all integrals in the set on the right-hand side are zero by assumption, $\mu$ vanishes on all open subsets of $\ts$. The outer regularity of $\mu$ at all Borel subsets of $\ts$ then implies that $\mu=0$.
	\end{proof}

\begin{remark}\quad
\begin{enumerate}
\item An alternative proof of \cref{res:embedding_of_M(X)} goes as follows. It is generally true that the norm dual $\nd{\vl}$ of a normed vector lattice $\vl$ is a vector sublattice of the order dual $\od{\vl}$ of $\vl$; see \cite[Theorem~30.8]{aliprantis_burkinshaw_PRINCIPLES_OF_REAL_ANALYSIS_THIRD_EDITION:1998}. Since $\nd{(\contctsR,\norm{\,\cdot\,}_\infty)}$ is (isometrically) lattice isomorphic to $\regularrealmeasts$,  it is now immediate that the map $\Ode$ in \cref{res:embedding_of_M(X)} is an injective lattice homomorphism.

This alternative approach uses the vector lattice part of the Riesz representation theorem, whereas our earlier proof does not.
\item There does not seem to be a result in the vein of \cref{res:embedding_of_M(X)} in \cite{bourbaki_INTEGRATION_VOLUME_I_CHAPTERS_1-6_SPRINGER_EDITION:2004}; presumably this is because the space $\regularrealmeasts$, which consists of measures in the sense of Caratheodory and Lebesgue, simply does not exist for Bourbaki.
\end{enumerate}
\end{remark}

We can now establish the following analogue of \cref{res:embedding_of_Lp_into_ordercontcRod}.

\begin{theorem}\label{res:embedding_of_M(X)_as_banach_lattice_with_properly_separated_elements}
	Let $\ts$ be a non-empty, locally compact space. For $\mu\in\regularrealmeasts $, set
\[
\pairing{\emb{\mu},f}\coloneqq\int_\ts\! f\di{\mu}
\]
for $f\in\contctsR$. Then $\emb{\mu}\in\contctsRod$, and the map $\Emb:\mu\mapsto\emb{\mu}$ defines an injective lattice homomorphism $\Emb:\regularrealmeasts\to\contctsRod$. Suppose that $h$ is a bounded Borel measurable function on $\ts$. Then $\emb{h\mu}=h\emb{\mu}$.

For $\mu\in\regularrealmeasts$, set $\norm{\emb{\mu}}\coloneqq \norm{\mu}$, thus making $\Emb(\regularrealmeasts)$ into a Dedekind complete Banach lattice. Then the set
\[
\lrdesset{\ode\in\Emb(\regularrealmeasts) : \supp\ode\text{ is compact and }\supp{\pos{\ode}}\cap\supp{\nega{\ode}}=\emptyset}
\]
is a dense subset of the Banach lattice $\Emb(\regularrealmeasts)$.
\end{theorem}

\begin{proof}
	It is clear that $\Emb$ is compatible with the multiplication by bounded Borel measurable functions. 
	In view of \cref{res:embedding_of_M(X)}, it is then only the density statement that requires proof.
	Let $\mu\in\regularrealmeasts$, and let $\mu=\pos{\mu}-\nega{\mu}$ be its decomposition into its positive and negative parts. There exists a partition of $\ts$ into disjoint Borel measurable subsets $\pos{\ts}$ and $\nega{\ts}$ of $\ts$ such that $\pos{\mu}(\nega{\ts})=0$, $\nega{\mu}(\pos{\ts})=0$, $\pos{\mu}(\pos{A})\geq 0$ for every Borel subset $\pos{A}$ of $\pos{\ts}$, and $\nega{\mu}(\nega{A})\geq 0$ for every Borel subset $\nega{A}$ of $\nega{\ts}$; see \cite[Theorem~6.14]{rudin_REAL_AND_COMPLEX_ANALYSIS_THIRD_EDITION:1987}. Let $\varepsilon>0$. Since $\pos{\mu}$, being finite, is inner regular at all Borel subsets of $\ts$ (see \cite[Proposition~7.5]{folland_REAL_ANALYSIS_SECOND_EDITION:1999}), there exists a compact subset $\pos{K}$ of $\pos{\ts}$ such that $0\leq\pos{\mu}(\pos{\ts})-\pos{\mu}(\pos{K})<\varepsilon/2$. Likewise, there exists a compact subset $\nega{K}$ of $\nega{\ts}$ with the property that $0\leq\nega{\mu}(\nega{\ts})-\nega{\mu}(\nega{K})<\varepsilon/2$.
	For $A\in\borelsalg$, we set $\pos{\nu}(A)\coloneqq\pos{\mu}(A\cap\pos{K})$ and $\nega{\nu}(A)\coloneqq\nega{\mu}(A\cap\nega{K})$, thus defining positive measures $\pos{\nu},\nega{\nu}$ on $\borelsalg$. Since $\pos{\mu}$ and $\nega{\mu}$ are \emph{finite} positive regular Borel measures,
	\cite[Section~7.2, Exercise~7]{folland_A_COURSE_IN_ABSTRACT_HARMONIC_ANALYSIS_SECOND_EDITION:2016} shows that $\pos{\nu},\nega{\nu}\in\regularrealmeasts$. Set $\nu\coloneqq\pos{\nu}-\nega
	\nu$. Then $\nu\in\regularrealmeasts$ and $\norm{\mu-\nu}<\varepsilon$. Furthermore, $\supp{\emb{\nu}}\subseteq \pos{K}\cup\nega{K}$ is a compact subset of $\ts$.
	
	Since $\supp\emb{\pos\nu}\subseteq\pos{K}$, $\supp\emb{\nega\nu}\subseteq\nega{K}$, and $\pos{K}\cap\nega{K}=\emptyset$, it follows that $\supp\emb{\pos\nu}\cap\supp\emb{\nega\nu}=\emptyset$. Hence \cref{res:spatial_and_lattice_disjointness} shows that $\emb{\pos\nu}$ and $\emb{\nega\nu}$ are disjoint elements of $\contctsRod$, and this implies that the equality $\emb{\nu}=\emb{\pos{\nu}}-\emb{\nega{\nu}}$ gives the decomposition of $\emb{\nu}$ in $\contctsRod$ into its positive and negative part $\pos{\emb\nu}$ and $\nega{\emb{\nu}}$, respectively. Since $\norm{\emb{\mu}-\emb{\nu}}=\norm{\mu-\nu}<\varepsilon$ by definition, the proof  of the theorem is complete.
\end{proof}

\section{Locally compact groups}\label{sec:locally_compact_groups}

\noindent
In this section, we shall review some material on locally compact groups and on Banach lattice and Banach lattice algebras on such groups. In particular, we shall describe various well-known Banach algebras that are studied within harmonic analysis. For details, see \cite{hewitt_ross_ABSTRACT_HARMONIC_ANALYSIS_VOLUME_I_SECOND_EDITION:1979,bourbaki_INTEGRATION_VOLUME_II_CHAPTERS_7-9_SPRINGER_EDITION:2004,folland_A_COURSE_IN_ABSTRACT_HARMONIC_ANALYSIS_SECOND_EDITION:2016,rudin_FOURIER_ANALYSIS_ON_GROUPS:1962}, and also \cite[Sections~3.3 and~4.5]{dales_BANACH_ALGEBRAS_AND_AUTOMATIC_CONTINUITY:2000}, for example.

\smallskip

A group that is also a locally compact space is a \emph{locally compact group} whenever the group operations are continuous.

Let $\group$ be a locally compact group. As for general locally compact spaces, the Borel $\sigma$-algebra of $\group$ will be denoted by $\borelsalg$. We shall write  $\measgroupR$ for $\meas(\group,\borelsalg,\RR)$ and  $\measgroupC$ for $\meas(\group,\borelsalg,\CC)$. There exists a non-zero, positive regular Borel measure $\Hm$ on $\group$ such that $\Hm(s\cdot A)=\Hm(A)$ for all $s\in\group$ and all Borel subsets $A$ of $\group$. Such a measure is a \emph{\uppars{left}\ Haar measure on $\group$}; it is unique up to a non-zero positive multiplicative constant. We shall write $\EllpgroupR$ for $\Ell^p(\group,\borelsalg,\Hm,\RR)$ and $\EllpgroupC$ for $\Ell^p(\group,\borelsalg,\Hm,\CC)$ .

Let $\group$ be a locally compact group, and let $\Hm$ be a Haar measure on $\group$. Then
\[
\int_\group\! f(as)\dH(s)=\int_\group\! f(s)\dH(s)
\]
for all $f\in\EllonegroupC$ and $a\in\group$. When $\group$ is abelian, the left Haar measure is trivially also right invariant, but this is not generally the case. There exists a continuous group homomorphism $\Delta:\group\to(0,\infty)$ such that
\[
\int_\group\! f(sa)\dH(s)=\Delta(a^{-1})\int_\group\! f(s)\dH(s)
\]
for all $f\in\EllonegroupC$ and $a\in\group$; some authors write $\Delta(a)$ where we use $\Delta(a^{-1})$.  The homomorphism $\Delta$ is the \emph{modular function of $\group$}. It is easy to see that $\Delta(s)=1$ for all $s\in\group$ when $\group$ is compact, so that the left Haar measure is then also right invariant.

\smallskip

Let $\group$ be a locally compact group. We recall from the general theory for locally compact spaces that the Banach lattice $\measgroupR$ is isometrically lattice isomorphic to the Banach lattice $\nd{\contogroupR}$. By combining this isomorphism with the group structure of the underlying locally compact space $\group$, a multiplication on $\measgroupR$ can be introduced such that it becomes a Banach lattice algebra. Take $\mu,\nu\in\measgroupR$. Then the \emph{convolution product $\mu\conv \nu$ of $\mu$ and $\nu$} is defined by
\begin{equation}\label{eq:convolution_definition}
\pairing{\mu\conv \nu,f}\coloneqq \int_\group\! \int_\group\!  f(st) \di{\mu(s)}\di{\nu(t)}\quad
\end{equation}
for all $f\in\contogroupR$. With this multiplication, $\measgroupR$ is a Banach lattice algebra. The unit mass at $e_\group$ is denoted by $\delta_{e_\group}$; it is the identity element of $\measgroupR$.
One can describe $\mu\conv\nu$ at the level of the Borel subsets of $\group$ by
\begin{equation}\label{eq:convolution_for_sets}
(\mu\conv\nu)(A)=\int_\group\! \nu(s^{-1} \cdot A)\di{\mu(s)}=\int_\group\! \mu(A \cdot s^{-1})\di{\nu(s)}
\end{equation}
for $A\in\borelsalg$.

The following basic result is very well known. Since it is essential to the results in Section~\ref{sec:main_theorem}, we nevertheless include the proof.

\begin{proposition}\label{res:support_of_convolution}
	Let $\group$ be a locally compact group, and take $\mu, \nu  \in\measgroupR$ with compact support.  Then $\supp{(\mu\conv \nu)}\subseteq\supp{\mu}\,\cdot\, \supp{\nu}$.
\end{proposition}

\begin{proof}
We may suppose that $\supp{\mu}\,\cdot\, \supp{\nu}\neq\group$. Then $\group\setminus (\supp{\mu}\,\cdot\, \supp{\nu})$ is a non-empty, open subset of $\group$. Take $f\in\contc(\group,\RR; \group\setminus (\supp{\mu}\,\cdot\, \supp{\nu}))$.  Then it is immediate from \cref{eq:convolution_definition} that $\pairing{\mu\conv\nu,f}=0$. Hence $\mu\conv\nu$ vanishes on $\group\setminus (\supp{\mu}\,\cdot\, \supp{\nu})$. The result follows.
\end{proof}

\smallskip

The complex Banach lattice $\measgroupC$ is the complexification of the Banach lattice $\measgroupR$. Since $\measgroupR$ is, in fact, a Banach lattice algebra, $\measgroupC$ is a complex Banach lattice algebra. It is then easily checked that the obvious complex analogues of \cref{eq:convolution_definition,eq:convolution_for_sets} hold.

Take $\mu \in\measgroupC$. Set  $\mu^*(A) = \overline{\mu(A^{-1})}$ for each Borel subset $A$ of $\group$. Then $\mu\mapsto \mu^*$ is an involution on $\measgroupC$.

The following theorem is basic; see \cite[Section~3.3]{dales_BANACH_ALGEBRAS_AND_AUTOMATIC_CONTINUITY:2000}.

\begin{theorem}
	Let $\group$ be a locally compact group. Then $\measgroupR$ is a Dedekind complete, unital Banach lattice algebra, and $\measgroupC$ is a unital, semisimple, complex Banach lattice $^\ast$-algebra. The identity element of both algebras is $\delta_{e_\group }$.
\end{theorem}

The commutativity of $\measgroupR$ and that of $\measgroupC$ are both equivalent to the group $\group$ being abelian.

\begin{remark}
In the literature on abstract harmonic analysis, the complex Banach lattice algebra $\measgroupC$ is usually denoted by $\meas(\group)$, and it is called the measure algebra of $\group$, without a reference to the complex field. It is then studied as a complex Banach $^\ast$-algebra. We, on the other hand, concentrate on the lattice properties of $\measgroupR$.
\end{remark}

Let $\group$ be a locally compact group with left Haar measure $\Hm$. The subspace of $\measgroupR$ consisting of all elements that are absolutely continuous with respect to $\Hm$ is a Banach sublattice of $\measgroupR$; it is also an algebra ideal and an order ideal of $\measgroupR$. This Banach sublattice is isometrically lattice isomorphic to $\EllonegroupR$ by using the Radon--Nikod{\'y}m theorem:
each $f \in\EllonegroupR$ corresponds to the measure $f{\dH}$ in $\measgroupR$. This identification provides $\EllonegroupR$ with a product; the convolution product of $f$ and $g$ in $\EllonegroupR$ is then given by the formulae
\begin{equation}\label{eq:convolution_of_two_ellone_functions}
(f\conv g)(t)= \int_\group  f(s)g(s^{-1}t)\dH(s)=  \int_\group  f(ts)g(s^{-1})\dH(s)
\end{equation}
for $m_\group$-almost all $t\in\group$.

Similar remarks apply to $\measgroupC$ and $\EllonegroupC$, with the additional feature that the subspace of $\measgroupC$ consisting of all elements that are absolutely continuous with respect to $\Hm$ is now an algebra $^\ast$-ideal. The identification of $\EllonegroupC$ with this subspace then provides $\EllonegroupC$ with an involution, denoted by $^\ast$ again. For $f\in\EllonegroupC$, the involution is given by
\[
f^\ast(s)=\overline{f(s^{-1})}\,\Delta(s^{-1})
\]
for $\Hm$-almost $s\in\group$.

We then have the following result.

\begin{theorem}
	Let $\group$ be a locally compact group. Then $\EllonegroupR$ is a Dedekind complete Banach lattice algebra which is a closed algebra ideal and an order ideal of $\measgroupR$, and $\EllonegroupC$ is a semisimple, complex Banach lattice $^\ast$-algebra which is a closed algebra $^\ast$-ideal of $\measgroupC$.
\end{theorem}

The commutativity of $\EllonegroupR$ and that of $\EllonegroupC$ are both equivalent to the group $\group$ being abelian. Both algebras have a positive contractive approximate identity, and both are unital if and only if $\group$ is discrete. In the latter case,  $\EllonegroupR=\measgroupR$ and $\EllonegroupC=\measgroupC$. It is then customary to write $\ell^1(\group,\RR)$ and $\ell^1(\group,\CC)$ for the coinciding convolution algebras over the respective fields.

We remark that the space $\EllonegroupR$ is not just an order ideal of $\measgroupR$, but that it is, in fact, a so-called \emph{band} of $\measgroupR$. More precisely, it is the band that is generated by $\Hm$. We have not defined what a band is in the present article, and we shall not pursue this matter further.

\begin{remark}
In the literature on abstract harmonic analysis, the complex Banach lattice algebra $\EllonegroupC$ is usually denoted by $\Ell^1(\group)$, and it is called the group algebra of $\group$,  without a reference to the complex field. It is then studied as a complex Banach $^\ast$-algebra, whereas we concentrate on the lattice properties of $\EllonegroupR$.
\end{remark}

\smallskip

Let $\group$ be a locally compact group, and take $p$ with $1\leq p<\infty$. Now take $\mu \in\measgroupF$ and $g\in\EllpgroupF$, and define
\begin{align}
(\mu\convp g)(s) &\coloneqq \int_\group\! g(t^{-1}s)\di{\mu(t)}\label{eq:mu_convp_g}
\intertext{and}
(g\convp \mu)(s)  &\coloneqq \int_\group\!  g(st^{-1}) \Delta_\group ^{1/p}(t^{-1})\di{\mu(t)}\label{eq:g_convp_mu}
\end{align}
for those $s\in\group$ for which these integrals exist; this can be shown to be $\Hm$-almost everywhere the case.

Now take $f\in\EllonegroupF$ and $g\in\EllpgroupF$. Identifying $f$ and $f\dH$,  \cref{eq:mu_convp_g,eq:g_convp_mu} specialise to
\begin{align}
(f\convp g)(s) &\coloneqq \int_\group\! f(t)g(t^{-1}s)(t)\dH(t)\label{eq:f_convp_g}
\intertext{and}
(g\convp f)(s)  &\coloneqq \int_\group\! g(st^{-1}) f(t) \Delta_\group ^{1/p}(t^{-1})\dH(t)\label{eq:g_convp_f}
\end{align}
for $\Hm$-almost all $s\in\group$.

The following theorem is contained in \cite[Section~3.3]{dales_BANACH_ALGEBRAS_AND_AUTOMATIC_CONTINUITY:2000}; see also \cite[(20.19)]{hewitt_ross_ABSTRACT_HARMONIC_ANALYSIS_VOLUME_I_SECOND_EDITION:1979}.

\begin{theorem}\label{res:action_of_measgroup_on_Ellp_basic_facts}
	Let $\group$ be a locally compact group, and take $p$ with $1\leq p<\infty$. Take $\mu \in\measgroupF$ and $g\in\EllpgroupF$. Then
	the functions $\mu \convp g$ and $g \convp  \mu$
	belong to $\EllpgroupF$, and $\norm{\mu\convp g}_p\leq\norm{\mu}\norm{g}_p$ and $\norm{f\convp \mu}_p\leq\norm{\mu}\norm{f}_p$.
\end{theorem}

The following is now clear.

\begin{corollary}\label{res:actions_on_ellp}
Let $\group$ be a locally compact group, and take $p$ with $1\leq p<\infty$.

For $\mu\in\measgroupF$, define $\rep_\mu:\EllpgroupF\to\EllpgroupF$ by setting
\[
\rep_\mu(g)\coloneqq\mu\convp g
\]
for $g\in\EllpgroupF$. Then $\rep_\mu\in\regularops{\EllpgroupF}$, and the map $\rep:\mu\mapsto\rep_\mu$ defines a positive Banach algebra homomorphism $\rep:\measgroupF\to\regularops{\EllpgroupF}$.
%

For $f\in\EllonegroupF$, define $\rep_f:\EllpgroupF\to\EllpgroupF$ by setting
\[
\rep_f(g)\coloneqq f\convp g
\]
for $g\in\EllpgroupF$. Then $\rep_f\in\regularops{\EllpgroupF}$, and the map $\rep:f\mapsto\rep_f$ defines a positive Banach algebra homomorphism $\rep:\EllonegroupF\to\regularops{\EllpgroupF}$.
%
\end{corollary}

Similarly, one can define a map $g\mapsto g\convp\mu$, respectively, $g\mapsto g\convp f$. Then the resulting map from $\measgroupF$, respectively, $\EllonegroupF$ into $\regularops{\EllpgroupF}$ has the same properties as its left-sided analogue, save that is an anti-homomorphism.
\smallskip

We shall see later that the two Banach algebra homomorphisms $\rep$ in \cref{res:actions_on_ellp} are both Banach lattice algebra homomorphisms; see \cref{res:action_of_measures_on_Ellp,res:action_of_Ellone_on_Ellp}, below.

\section{Locally compact semigroups}\label{sec:locally_compact_semigroups}

\noindent
In this section, we shall collect some material on Banach lattice algebras on locally compact semigroups. It is for these algebras that we shall benefit from the results in Section~\ref{sec:closed_subspaces_of_locally_compact_spaces} by using them in the proof of our main result, \cref{res:general_result_for_semigroups}, below.

\smallskip

\begin{definition}\label{def:weight_on_semigroup}
Let $S$ be a locally compact semigroup.  A \emph{weight on $S$} is a continuous function $\omega: S \to (0,\infty)$ such that
\[
\omega(st)\leq \omega(s)\omega(t)
\]
for all $s,t\in S$.
\end{definition}

Let $\group$ be a locally compact group, and let $S$ be a closed subspace of $\group$ that is a subsemigroup of $\group$. Suppose that $\omega$ is a weight on $S$, and consider the subset $\meas(S,\omega,\RR)$ of $\meas(S,\RR)$ consisting of all elements $\mu$ of $\meas(S,\RR)$ such that
\[
\int_S\! \omega(t)\di{\abs{\mu}(t)}< \infty.
\]
Then $\meas(S,\omega,\RR)$ is a Dedekind complete vector sublattice of $\meas(S,\RR)$. (It is, in fact, even a band in $\meas(S,\RR)$.) Since $\meas(S,\RR)$ can be embedded as a sublattice of $\od{\contc(S,\RR)}$ by \cref{res:embedding_of_M(X)_as_banach_lattice_with_properly_separated_elements}, this is also the case for the sublattice  $\meas(S,\omega,\RR)$ of $\meas(S,\RR)$. Since, furthermore, $\od{\contc(S,\RR)}$ can be embedded as a sublattice of $\od{\contc(\group,\RR)}$ by \cref{res:embedding_of_order_duals}, we see that $\meas(S,\omega,\RR)$ can be embedded as a sublattice of $\od{\contc(\group,\RR)}$. The embedded copy is easily checked to be a subalgebra of $\measgroupR$, and hence the embedding of $\meas(S,\omega,\RR)$ into $\od{\contc(\group,\RR)}$ provides  $\meas(S,\omega,\RR)$ with a (convolution) product.

We introduce a norm $\norm{\,\cdot\,}_\omega$ on $\meas(S,\omega,\RR)$ by setting
\[
\norm{\mu}_\omega\coloneqq \int_S\! \omega(t)\di{\abs{\mu}(t)}
\]
for $\mu\in\meas(S,\omega,\RR)$. Then $(\meas(S,\omega,\RR), \norm{\,\cdot\,}_\omega)$ is a Banach algebra. The algebra $(\meas(S,\omega,\RR), \norm{\,\cdot\,}_\omega)$ is called a \emph{Beurling algebra}. It is a Dedekind complete Banach lattice algebra.

We then have the following companion result of \cref{res:embedding_of_M(X)_as_banach_lattice_with_properly_separated_elements}.

\begin{theorem}\label{res:embedding_of_beurling_algebras_of_measures_as_banach_lattice_with_properly_separated_elements}
	Let $\group$ be a locally compact group, and let $S$ be a closed subspace of $\group$ that is a subsemigroup of $\group$. Suppose that $\omega$ is a weight on $S$. For each  $\mu\in\meas(S,\omega,\RR)$, set
	\[
	\pairing{\emb{\mu},f}\coloneqq\int_S\! f\di{\mu}
	\]
	for $f\in\contcgroupR$. Then the map $\Emb:\mu\mapsto\emb{\mu}$ defines an injective lattice homomorphism $\Emb:\meas(S,\omega,\RR)\to\od{\contc(\group,\RR)}$. Suppose that $h$ is a bounded Borel measurable function on $S$, and extend $h$ to a Borel measurable function $\widetilde h$ on $\group$ by setting $\widetilde h(t)\coloneqq 0$ for $t\in\group\setminus S$. Then $\emb{h\mu}=\widetilde{h}\emb{\mu}$.
	
	For $\mu\in\meas(S,\omega,\RR)$, set $\norm{\emb{\mu}}\coloneqq \norm{\mu}_\omega$, thus making $\Emb(\meas(S,\omega,\RR))$ into a Dedekind complete Banach lattice. Then the set
	\[
	\lrdesset{\ode\in\Emb(\meas(S,\omega,\RR)) : \supp\ode\text{ is compact and }\supp{\pos{\ode}}\cap\supp{\nega{\ode}}=\emptyset}
	\]
	is a dense subset of the Banach lattice $\Emb(\meas(S,\omega,\RR))$.
\end{theorem}

\begin{proof}
	In view of the above, all is clear except the density statement. For this, let $\mu\in\meas(S,\omega,\RR)$. Since $\omega$ is strictly positive and continuous, the measure $\omega\di{\abs{\mu}}$ is a positive regular Borel measure on $S$; see \cite[Section~7.2, Exercise~9]{folland_REAL_ANALYSIS_SECOND_EDITION:1999}. An easy modification of the argument in the proof of \cref{res:embedding_of_M(X)_as_banach_lattice_with_properly_separated_elements} then shows that the subset
\[
\desset{\mu\in\meas(S,\omega,\RR): \supp{\mu}\text{ is compact and } \supp\pos{\mu}\cap\supp{\nega
	\mu}=\emptyset}
\]
is a dense subset of $\meas(S,\omega,\RR)$. As in the proof of \cref{res:embedding_of_M(X)_as_banach_lattice_with_properly_separated_elements}, the density statement for the embedded copy $\Emb(\meas(S,\omega,\RR))$ of $\meas(S,\omega,\RR)$ is then immediate.
\end{proof}

Let $\group$ be a locally compact group, and let $S$ be a closed subspace of $\group$ that is a subsemigroup of $\group$. Suppose that $\omega$ is a weight on $S$. It is obvious how to define the complex analogue $\meas(S,\omega,\CC)$ of $\meas(S,\omega,\RR)$. Then $\meas(S,\omega,\CC)$ is the complexification of $\meas(S,\omega,\RR)$; hence $\meas(S,\omega,\CC)$ is a complex Banach lattice algebra.

\smallskip

Let $S$ be a semigroup, supplied with the discrete topology, and let $\omega$ be a weight on $S$.
Instead of considering real-valued measures on $S$ as above, we now consider $\ell^1$-spaces for weighted counting measures, as follows. Let $\ell^1(S,\omega,\RR)$ consist of the functions $f:S\to\RR$ such that
\begin{equation}\label{eq:discrete_beurling_one}
 \sum_{s\in S} \abs{f(s)}\,\omega(s)< \infty.
\end{equation}
We introduce a norm $\norm{\,\cdot\,}_\omega$ on $\ell^1(S,\omega,\RR)$ by setting
\begin{equation}\label{eq:discrete_beurling_two}
\norm{f}_\omega\coloneqq \sum_{s\in S} \abs{f(s)}\,\omega(s)
\end{equation}
for $f\in\ell^1(S,\omega)$. Then $(\ell^1(S,\omega,\RR),\norm{\,\cdot\,}_\omega)$ is a Banach space. For $s\in S$, we let $\delta_s$ denote the characteristic function of the subset $\{s\}$ of $S$. Then there is a unique continuous product on $\ell^1(S,\omega,\RR)$ such that  $\delta_{s_1} \conv\delta_{s_2} = \delta_{s_1s_2}$ for $s_1,s_2\in S$. When supplied with the pointwise ordering, the weighted $\ell^1$-space $(\ell^1(S,\omega,\RR),\norm{\,\cdot\,}_\omega)$ is then a Dedekind complete Banach lattice algebra, which is also called a \emph{Beurling algebra}.

Let $S$ be a semigroup, supplied with the discrete topology, and let $\omega$ be a weight on $S$.
It is obvious how to use \cref{eq:discrete_beurling_one,eq:discrete_beurling_two} to define the complex analogue $\ell^1(S,\omega,\CC)$ of $\ell^1(S,\omega,\RR)$. Then $\ell^1(S,\omega,\CC)$ is the complexification of $\ell^1(S,\omega,\RR)$. Hence $\ell^1(S,\omega,\CC)$ is a complex Banach lattice algebra.

\smallskip

Let $\group$  be a group, supplied with the discrete topology, and let $\omega$ be a weight on $\group$.  Then it is a notorious open question whether the Beurling algebra $\ell^{\,1}(\group ,\omega,\CC)$ is always semisimple.
It is proved in \cite[Theorem~7.13]{dales_lau:2005} that this is the case whenever $\group$ is a maximally almost periodic group and $\omega$ is an arbitrary weight
on $\group$, and also whenever $\group$ is an arbitrary group and $\omega$ is a symmetric weight on $\group$, in the sense that $\omega(s^{-1})=\omega(s)$ for $s\in\group$.

For semigroups, however, it is known that such Beurling algebras need not be semisimple. They can even be radical, as we shall now indicate.

Let $S$ be a semigroup, supplied with the discrete topology, and let $\omega$ be a weight on $S$. For $s\in S$, the element $\delta_s$ of the Beurling algebra $\ell^1(S,\omega,\CC)$ is obviously quasi-nilpotent if and only if
\[
\lim_{n\to\infty}\omega(s^n)^{1/n}=0.
\]
It is shown in \cite[Example 2.3.13(ii)]{dales_BANACH_ALGEBRAS_AND_AUTOMATIC_CONTINUITY:2000} that $\ell^{\,1}(S, \omega,\CC)$ is a radical
Banach algebra whenever $\delta_s$ is quasi-nilpotent for all $s\in S$ and $\omega(st)=\omega(ts)$ for all $s,t\in S$. For example, take
$S=\ZZ^+$ and set $\omega(n) \coloneqq \exp(-n^2)$ for $n\in\ZZ^+$, or take $S$  to be the free semigroup on two generators and set
$\omega(w) = \exp(-\abs{w}^2)$  for a word $w$ in $S$, where $\abs{w}$ is the length of the word $w$. Then in both cases
$\ell^{\,1}(S, \omega,\CC)$ is a radical Banach algebra.

For a study  of the algebras $\ell^{\,1}(S, \omega,\CC)$ when $S$ is a subsemigroup of $\RR$, see \cite{dales_dedania:2009}.  In the case where $S=\ZZ^+$, the algebras
$\ell^{\,1}(\ZZ^+, \omega,\CC)$ are examples of {\it Banach algebras of  power series}; for a study of these algebras, see \cite{bade_dales:1989,dales_BANACH_ALGEBRAS_AND_AUTOMATIC_CONTINUITY:2000}.

\smallskip

We shall now consider continuous analogues of the Beurling algebras $\ell^1(S,\omega,\RR)$ and $\ell^1(S,\omega,\CC)$ above.

Consider the unital additive semigroup $\pos{\RR}\coloneqq[0,\infty)$. Suppose that $\omega$ is a weight on $\pos{\RR}$. Then we define $\Ell^1(\pos{\RR}, \omega,\RR)$  to be the vector space of measurable functions $f$ on $\pos{\RR}$ such that
\begin{equation}\label{eq:half_line_one}
\int_{\pos{\RR}} \!\!\abs{f(t)}\, \omega(t)\di{t} < \infty,
\end{equation}
and we introduce a norm $\norm{\,\cdot\,}_\omega$ on  $\Ell^1(\pos{\RR}, \omega,\RR)$ by setting
\begin{equation}\label{eq:half_line_two}
\norm{f}_\omega \coloneqq\int_{\pos{\RR}}\!\! \abs{f(t)}\, \omega(t)\di{t}
\end{equation}
for  $f\in\Ell^1(\pos{\RR}, \omega,\RR)$. For $f,g\in\Ell^1(\pos{\RR}, \omega,\RR)$, we set
\begin{equation}\label{eq:half_line_three}
(f\conv g)(s)\coloneqq\int_{[0,s]} f(t)g(s-t)\di{t}
\end{equation}
for all $s\in\pos{\RR}$ for which the integral exists, which can be shown to be the case almost everywhere. With this (convolution) product, $(\Ell^1(\pos{\RR}, \omega,\RR),\norm{\,\cdot\,}_\omega)$ is a Dedekind complete Banach lattice algebra that is again an example of a Beurling algebra.

It is obvious how to use \cref{eq:half_line_one,eq:half_line_two,eq:half_line_three} to define the complex analogue $\Ell^1(\pos{\RR},\omega,\CC)$ of $\Ell^1(\pos{\RR},\omega,\RR)$. Then $\Ell^1(\pos{\RR},\omega,\CC)$ is the complexification of $\Ell^1(\pos{\RR},\omega,\RR)$. Hence $\Ell^1(\pos{\RR},\omega,\CC)$ is a complex Banach lattice algebra.

The Beurling algebras $\Ell^1(\pos{\RR}, \omega,\CC)$ are studied in \cite{bade_dales:1981,dales_BANACH_ALGEBRAS_AND_AUTOMATIC_CONTINUITY:2000}, for example. It can
be shown that $\rho_\omega\coloneqq\lim_{t\to \infty} \omega(t)^{1/t}$ always exists, that $\Ell^1(\pos{\RR}, \omega,\CC)$ is semisimple if $\rho_\omega>0$, and that $\Ell^1(\pos{\RR}, \omega,\CC)$ is radical if $\rho_\omega=0$. For example, the weight $\omega :t\mapsto \exp(-t^2)$ gives a radical Beurling algebra on $\pos{\RR}$.

Let $\omega$ be a weight on $\pos{\RR}$. Then it follows from Titchmarsh's convolution theorem \cite[Theorem 4.7.22]{dales_BANACH_ALGEBRAS_AND_AUTOMATIC_CONTINUITY:2000} that the Beurling algebras $\Ell^1(\pos{\RR}, \omega,\RR)$ and $\Ell^1(\pos{\RR}, \omega,\CC)$ are integral domains.

\section{Main theorem}\label{sec:main_theorem}

\noindent
In this section, we shall establish our main result, \cref{res:general_result_for_semigroups}, below, in the context of non-empty, closed semigroups in locally compact groups, as well as a related, easier, result in the context of discrete semigroups.

\smallskip

We start with the following preparatory result.

\begin{lemma}\label{res:separation_lemma}
	Let $\group$ be a locally compact group, and let $K_1$ and $K_2$ be non-empty, disjoint, compact subsets of $\group$. Then there exists an open neighbourhood $U$ of $e_\group$ such that $K_1U$ and $K_2U$ are disjoint subsets of $\group$.
\end{lemma}

\begin{proof}
Since $\group$ is locally compact, there exists open neighbourhoods $W_1$ and $W_2$ of $K_1$ and $K_2$, respectively, such that $\overline {W_1}$ and $\overline{W_2}$ are disjoint. It is easy to see that, for $i=1,2$, there is an open neighbourhood $U_i$ of $e_\group$ with $K_iU_i\subseteq W_i$. Set $U\coloneqq U_1\cap U_2$. Then $U$ is an open neighbourhood of $e_\group$ and we also see that $K_1U\cap K_2 U\subseteq K_1U_1\cap K_2U_2\subseteq W_1\cap W_2=\emptyset$.
\end{proof}

\smallskip

For the ease of formulation, we introduce the following terminology. 

\begin{definition}
Let $\ode\in\contctsRod$. Then \emph{the support of $\ode$ is separated}, or \emph{$\ode$ has separated support},  if the supports of $\pos{\ode}$ and $\nega{\ode}$ are disjoint subsets of $\ts$.	
\end{definition}

We now come to our main result, \cref{res:general_result_for_semigroups}. It employs a notation $\bil$ for a bilinear map $\bil$ that suggests convolution, without requiring that this actually be the case. The reason is that we also want to cover situations where, for example, $\mu$ is a positive measure on the Borel $\sigma$-algebra of a locally compact group $\group$ and a function $f_1$ acts on a function $f_2$ via the formula
\begin{equation}\label{eq:weighted_convolution}
f_1\bil f_2(s)=\int_\group \! f_1(t)f_2(t^{-1}s)\di{\mu(t)}
\end{equation}
for $\mu$-almost all $s\in\group$. Unless the measure $\mu$ is a left Haar measure on $\group$, this is not an actual convolution, but obviously it still satisfies the relation
\[
\desset{s\in\group: (f_1\bil f_2)(s)\neq 0}\subseteq\desset{s\in\group: f_1(s)\neq 0}\,\cdot\,\desset{s\in\group: f_2(s)\neq 0},
\]
which is akin to the inclusion relation in the crucial first clause of the hypotheses in \cref{res:general_result_for_semigroups}. Such bilinear maps occur in \cite{oztop_samei:2017,oztop_samei:UNPUBLISHED}, for example, and \cref{res:general_result_for_semigroups} is likely to be applicable in such contexts.

\begin{theorem}\label{res:general_result_for_semigroups}
	
	Let $\group$ be a locally compact group, and let $S$ be a non-empty, closed subspace of $\group$ that is a subsemigroup of $\group$. Let $X$, $Y$\!, and $Z$ be vector sublattices of $\od{\contc(S,\RR)}$ that are Banach lattices, and where $Z$ is Dedekind complete.
	
	 Suppose that $\bil:X\times Y\to Z$ is a bilinear map such that $x\bil y\in\pos{Z}$ whenever  $x\in\pos{X}$ and $y\in\pos{Y}$. Define the positive linear map $\rep:X\to\regular(Y,Z)$ by $\rep_x(y)\coloneqq x\bil y$ for $x\in X$ and $y\in Y$.
	
	Suppose that the following conditions are satisfied:
	\begin{enumerate}
		\item $\supp{(x\bil y)}\subseteq\supp{x}\cdot\supp{y}$ for all $x\in \pos{X}$ and $y\in \pos{Y}$ with compact support;
		\item the elements of $X$ with compact, separated support are dense in $X$;
		\item the elements of $\pos{Y}$ with compact support are dense in $\pos{Y}$;
		\item $\ind_A y$ is an element of $Y$ again, whenever $y\in\pos{Y}$ has compact support and $A$ is a Borel subset of  $\supp{y}$.
	\end{enumerate}
Then $\rep$ is a lattice homomorphism.
\end{theorem}

For the sake of clarity, we recall that semigroups are not supposed to be unital.

\begin{proof}

We start with the case where $S=\group$. We are to prove that $\abs{\rep_x}=\rep_{\abs{x}}$ for all $x\in X$.

Recalling that positive linear maps between Banach lattices are continuous, that $\regular(Y,Z)$ is a Banach lattice in the regular norm, and that the modulus is continuous on Banach lattices, we see that the maps $x\mapsto \abs{\rep_x}$ and $x\mapsto\rep_{\abs{x}}$ are both continuous maps from $X$ into $\regular(Y,Z)$. By density, it is thus sufficient to prove that $\abs{\rep_x}=\rep_{\abs{x}}$ for all elements $x$ of $X$ with separated, compact support.

For this, we need to show that $\abs{\rep_x}(y)=\rep_{\abs{x}}(y)$ for all $y\in Y$. It is sufficient to establish this for all $y\in\pos{Y}$. By the continuity of the regular operators $\abs{\rep_x}$ and $\rep_{\abs{x}}$ on the Banach lattice $Y$, it is, by density, sufficient to prove that $\abs{\rep_x}(y)=\rep_{\abs{x}}(y)$ for all elements of $\pos{Y}$ with compact support.

All in all, we see that it is sufficient to demonstrate that $\abs{\rep_x}(y)= \rep_{\abs{x}}(y)$, whenever $x$ is an element of $X$ with separated, compact support and $y$ is an element of  $\pos{Y}$ with compact support.

In order to do so, we fix an element $x$ of $X$ with compact, separated support, and we let $x=\pos{x}-\nega{x}$ be the decomposition of $x$ into its disjoint positive and negative parts. The supports of $\pos{x}$ and $\nega{x}$ are disjoint, compact subsets of $X$. Using \cref{res:separation_lemma}, we can then choose and fix a relatively compact open neighbourhood $U$ of the $e_\group$ such that $(\supp{\pos{x}})U\cap(\supp{\nega{x}})U=\emptyset$.

We shall now first consider the special case in which the support of $y\in\pos{Y}$ is not only compact, but where it is also `sufficiently small'. To be precise, suppose that $y$ is an element of $\pos{Y}$ with compact support such that $\supp{y}\subseteq Us$ for some $s\in \group $. We shall show that then $\abs{\rep_x}(y)=\rep_{\abs{x}}(y)$.

First, since $\rep$ is positive, it is automatic that $\abs{\rep_x}\leq \rep_{\abs{x}}$, so that we have $\abs{\rep_x}(y)\leq \rep_{\abs{x}}(y)$.

Second, for the reverse inequality, we notice that certainly $\abs{\rep_x}(y)\geq \pm\rep_x(y)$, so that $\abs{\rep_x}(y)\geq\abs{\rep_x(y)}$. Since
\[\supp{(\pos{x}\bil y)}\subseteq\supp{\pos{x}}\cdot\,\supp{y}\subseteq\supp{\pos{x}}\cdot\, Us
\]
and
\[
\supp{(\nega{x}\bil y)}\subseteq\supp{\nega{x}}\cdot\,\supp{y}\subseteq\supp{\nega{x}}\cdot\, Us,
\]
the supports of $\pos{x}\bil y$ and $\nega{x}\bil y$ are still disjoint subsets of $\ts$. \cref{res:spatial_and_lattice_disjointness}, therefore, implies that
\begin{align*}
\abs{\rep_x(y)}&=\abs{\pos{x}\bil y-\nega{x}\bil y}=\abs{\pos{x}\bil y}+\abs{-\nega{x}\bil y}\\
&=\pos{x}\bil y + \nega{x}\bil y=\abs{x}\bil y=\rep_{\abs{x}}(y).
\end{align*}
We conclude that $\abs{\rep_x}(y)\geq\rep_{\abs{x}}(y)$.
We have established that $\abs{\rep_x}(y)=\rep_{\abs{x}}(y)$ in the special case where $y\in\pos{Y}$ is such that $\supp{y}$ is contained in $Us$ for some $g\in\group$.

Now suppose that $y$ is an arbitrary element of $\pos{Y}$ with compact support. Choose an open neighbourhood $V$ of $e_\group$ such that $\overline V\subseteq U$. Then $\supp{y}$ is contained in a union of finitely many right translates of $V$. Since $\ind_A y$ is still in $Y$ for all Borel subsets $A$ of $\supp{y}$, it is then easy to see that $y$ is a finite sum of elements of $\pos{Y}$, each of which is supported in a right translate of $\overline{V}$, hence in a right translate of $U$. By linearity, it follows from the result as established for the special case that $\abs{\rep_x}(y)=\rep_{\abs{x}}(y)$.

We have now established the theorem in the case where $S=\group$.

Next, we turn to the case of a general closed subspace $S$ of $\group$ that is a subsemigroup of $\group$.
The problem with the above proof in this case is that translates of open subsets need not be open again. Even if $S$ is unital, the proof of \cref{res:separation_lemma} breaks down, as does the argument in the final paragraph for the group case.

In order to circumvent this, we use \cref{res:embedding_of_order_duals} to embed $\od{\contc(S,\RR)}$ as a vector sublattice of $\od{\contc(\group,\RR)}$. By restriction, this global embedding yields embeddings of $X$, $Y$\!, and $Z$ as vector sublattices  $X^{\#}$\!, $Y^{\#}$\!, and $Z^{\#}$\! of $\od{\contc(\group,\RR)}$. By transporting the norms, these vector sublattices $X^{\#}$\!, $Y^{\#}$\!, and $Z^{\#}$\! then become Banach lattices. The bilinear map $\bil:X\times Y\to Z$ yields a bilinear map $\bil^{\#}:X^{\#}\times Y^{\#}\to Z^{\#}$. Since \cref{res:embedding_of_order_duals} also states that supports are preserved under the embedding of $\od{\contc(S,\RR)}$ into $\od{\contc(\group,\RR)}$, it is then immediate that the hypotheses in the theorem are satisfied for the sublattices $X^{\#}$\!, $Y^{\#}$\!, and $Z^{\#}$\! of $\od{\contc(\group,\RR)}$ and the bilinear map $\bil^{\#}$. We can now apply the result for the group case to these data. By transport of structure in the reverse direction, the result for the semigroup case then follows.
\end{proof}

During the above proof, it was indicated why the argument for the case of groups cannot in general be directly applied to the case of arbitrary semigroups. A closer inspection, however, shows that the argument for the case of groups \emph{is} valid in the case of a semigroup that is discrete and cancellative.  We recall that a semigroup $S$ is \emph{cancellative} if the maps $s\mapsto st$ and $s\mapsto ts$  from $S$ to $S$ are both injective for each $t \in S$. We shall now indicate the ingredients for the proof in this case.

Suppose that $S$ is a cancellative semigroup, supplied with the discrete topology. Then $\contc(S,\RR)$ consists of the real-valued functions with finite support, and $\od{\contc(S,\RR)}$ can be identified as a vector lattice with the real-valued functions on $S$. Consequently, $\ode$ has separated support for all $\ode\in\od{\contc(S,\RR)}$.
Suppose that $x\in\od{\contc(S,\RR)}$, and let $x=\pos{x}-\nega{x}$ be the decomposition of $x$ into its disjoint positive and negative parts. Then $\supp{\pos{x}}\,\cdot\,s$ and $\supp{\nega{x}}\,\cdot\,s$ are disjoint subsets of $S$ for all $s\in S$, due to the fact that $S$ is cancellative.
Finally, if $y\in\od{\contc(S,\RR)}$ has compact support, then $y$ is a finite sum of elements of $\od{\contc(S,\RR)}$, each of which is supported in a subset $\{s\}$ of $S$ for some $s\in S$.

After these preliminary remarks, the reader will have no difficulty verifying the following result along the lines of the proof of \cref{res:general_result_for_semigroups}. 

\begin{theorem}\label{res:general_result_for_discrete_semigroups}
	
	Let $S$ be a cancellative semigroup, supplied with the discrete topology. Let $X$, $Y$\!, and $Z$ be vector sublattices of $\od{\contc(S,\RR)}$ that are Banach lattices, and where $Z$ is Dedekind complete.
	
	Suppose that $\bil:X\times Y\to Z$ is a bilinear map such that $x\bil y\in\pos{Z}$ whenever  $x\in\pos{X}$ and $y\in\pos{Y}$. Define the positive linear map $\rep:X\to\regular(Y,Z)$ by $\rep_x(y)\coloneqq x\bil y$ for $x\in X$ and $y\in Y$.
	
	Suppose that the following conditions are satisfied:
	\begin{enumerate}
		\item $\supp{(x\bil y)}\subseteq\supp{x}\cdot\supp{y}$ for arbitrary $x\in \pos{X}$ and for all $y\in \pos{Y}$ with finite support;
		\item the elements of $\pos{Y}$ with finite support are dense in $\pos{Y}$;
		\item $\ind_{\{s\}} y$ is an element of $Y$ again, whenever $y\in\pos{Y}$ has finite support and $s\in S$.
	\end{enumerate}
	Then $\rep$ is a lattice homomorphism.
\end{theorem}

Naturally, \cref{res:general_result_for_discrete_semigroups} follows from \cref{res:general_result_for_semigroups} for all semigroups that are subsemigroups of groups. It is known that every abelian cancellative semigroup is a subsemigroup of a group, in which case the enveloping group can even be taken to be of the same cardinality as $S$; see \cite[Proposition 1.2.10]{dales_BANACH_ALGEBRAS_AND_AUTOMATIC_CONTINUITY:2000}.  In general, however, a unital cancellative semigroup is not necessarily a subsemigroup of any group; necessary and sufficient conditions for this, and examples where the conditions fail, are given in  \cite[Chapter~10]{clifford_preston_THE_ALGEBRAIC_THEORY_OF_SEMIGROUPS_II}. This shows that \cref{res:general_result_for_discrete_semigroups} has value independent of \cref{res:general_result_for_semigroups}.

\section{Lattice homomorphisms in harmonic analysis}\label{sec:lattice_homomorphisms_in_harmonic_analysis}

\noindent All material is now in place to show that a number of natural positive maps in harmonic analysis are, in fact, lattice homomorphisms. For each of them, all that needs to be done is merely to establish that \cref{res:general_result_for_semigroups} or \cref{res:general_result_for_discrete_semigroups} is applicable in the relevant context. In view of the general nature of these two results, it seems not unlikely that they can have future applications to cases that are not covered in the present section.

\smallskip

Our first result answers the original question mentioned in Section~\ref{sec:introduction_and_overview}, which is \cite[Problem~7]{wickstead:2017c}.

\begin{theorem}\label{res:left_regular_representation_of_measures}
	Let $\group$ be a locally compact group. Then the left regular representation $\leftreg:\measgroupF\to\regularops{\measgroupF}$ is an isometric Banach lattice algebra homomorphism.
\end{theorem}

\begin{proof}
	Since $\measgroupF$ is a unital Banach algebra in which the identity element has norm one, it is clear that $\leftreg$ is an isometric Banach algebra homomorphism. It remains to be shown that $\leftreg$ is a lattice homomorphism.
	
	For this, we start with the case where $\FF=\RR$. Then  \cref{res:embedding_of_M(X)_as_banach_lattice_with_properly_separated_elements} shows that $\measgroupR$ can be embedded as a sublattice of $\contcgroupRod$ via a map $\mu\mapsto\emb{\mu}$ for $\mu\in\measgroupR$, and that, after transport of the norm, the embedded sublattice is a Dedekind complete Banach lattice $X$ such that its elements with compact, separated support are dense. We now resort to \cref{res:general_result_for_semigroups}, where we take $Y$ and $Z$ to be equal to $X$. Then the clauses (2) and (3) of the hypotheses of \cref{res:general_result_for_semigroups} are satisfied, and it is easy to see that clause (4) of these hypotheses is also satisfied. Furthermore, \cref{res:support_of_convolution} shows that clause (1) of the hypothesis of \cref{res:general_result_for_semigroups} is also satisfied when we set
	\[
	\emb{\mu}\bil\emb{\nu}\coloneqq\emb{\mu\conv\nu}
	\]
	for $\mu,\nu\in\measgroupR$.
	An appeal to \cref{res:general_result_for_semigroups} concludes the proof for the case where $\FF=\RR$.
	
	As explained earlier, the complex case follows from the real case on general grounds because the complex Banach lattice algebra $\measgroupC$ is the complexification of the Banach lattice algebra $\measgroupR$.
\end{proof}

Our next result concerns the action of $\measgroupF$ on $\EllpgroupF$ for $1\leq p<\infty$ as defined in \cref{eq:mu_convp_g}. As announced in the beginning of this section, the proof is quite similar to that of \cref{res:left_regular_representation_of_measures}.

\begin{theorem}\label{res:action_of_measures_on_Ellp}
		Let $\group$ be a locally compact group, and take $p$ with $1\leq p<\infty$. For $\mu\in\measgroupF$ and $g\in\EllpgroupF$, define
		\begin{equation*}
		(\rep_\mu g)(s) \coloneqq \int_\group\! g(t^{-1}s)\di{\mu(t)}
		\end{equation*}
		for those $\Hm$-almost $s\in G$ for which the integral exists. Then $\rep_\mu g\in\EllpgroupF$ for all $\mu\in\measgroupF$ and $g\in\EllpgroupF$, $\rep_\mu$ is a regular operator on $\EllpgroupF$ for all $\mu\in\measgroupF$, and the map $\mu\mapsto\rep_\mu$ is an injective Banach lattice algebra homomorphism $\rep:\measgroupF\to\regularops{\EllpgroupF}$.
\end{theorem}

\begin{proof} All statements in the theorem are well known, except the one that states that $\rep$ is a lattice homomorphism. The proof for this has a general outline that is similar to that of \cref{res:left_regular_representation_of_measures}.
	
Again, we start with the case where $\FF=\RR$. In this case, \cref{res:embedding_of_M(X)_as_banach_lattice_with_properly_separated_elements}
shows again that $\measgroupR$ can be embedded as a sublattice of $\contcgroupRod$ by means of a map $\mu\mapsto\emb{\mu}$ for $\mu\in\measgroupR$ and that, after transport of the norm, the embedded sublattice is a Banach lattice $X$ such that its elements with compact, separated support are dense. Furthermore, \cref{res:embedding_of_Lp_into_ordercontcRod} shows that $\EllpgroupR$ can be embedded as a sublattice of $\contcgroupRod$ via a map $g\to\emb{g}$ and that, after transport of the norm, the embedded sublattice is a Dedekind complete Banach lattice $Y$ such that its elements with compact support are dense. We now resort to \cref{res:general_result_for_semigroups}, where we take $Z$ to be equal to $Y$, and where we set
\[
\emb{\mu}\bil\emb{g}\coloneqq\emb{\rep_\mu(g)}
\]
for $\mu\in\measgroupR$ and $g\in\EllpgroupR$.
Then the hypotheses of \cref{res:general_result_for_semigroups} are satisfied, and an application of this theorem concludes the proof for the case where $\FF=\RR$.

As explained earlier, the complex case follows from the real case on general grounds because the complex Banach lattice $\measgroupC$ is the complexification of the Banach lattice  $\measgroupR$ and the complex Banach lattice $\EllpgroupC$ is the complexification of the Banach lattice $\EllpgroupR$.
\end{proof}

\begin{remark}
The action of $\measgroupF$ on $\EllpgroupF$ by convolution was previously studied by Arendt, using earlier results by Brainerd and Edwards (see \cite{brainerd_edwards:1966a}) and Gilbert (see \cite{gilbert:1968}) on convolutions. In \cite{arendt:1981}, Arendt showed that the map $\rep$ in \cref{res:action_of_measures_on_Ellp} is a lattice homomorphism when $p=1$, and also when $1<p<\infty$ and $\group$ is amenable. As \cref{res:action_of_measures_on_Ellp} shows, for $1<p<\infty$, the assumption that $\group$ be amenable is redundant.

We shall discuss Arendt's approach in more detail in Section~\ref{sec:possible_further_research_in_ordered_harmonic_analysis}.
\end{remark}

Since $\EllpgroupF$ is a Banach lattice subalgebra of $\measgroupF$, we have the following consequence of \cref{res:action_of_measures_on_Ellp}, where the action of $\EllonegroupF$ on $\EllpgroupF$ is now given by \cref{eq:f_convp_g}. It solves \cite[Problem~6]{wickstead:2017c}. It can also be found in earlier work by Kok (see \cite[Example~7.6]{kok_UNPUBLISHED:2016}), where it was obtained via a different approach.

\begin{corollary}\label{res:action_of_Ellone_on_Ellp}
		Let $\group$ be a locally compact group, and take $p$ with $1\leq p<\infty$. For $f\in\EllonegroupF$ and $g\in\EllpgroupF$, define
		\begin{equation*}
		(\rep_f g)(s) \coloneqq \int_\group\! f(t)g(t^{-1}s)\dH(t)
		\end{equation*}
		for those $\Hm$-almost $s\in G$ for which the integral exists. Then $\rep_f g\in\EllpgroupF$ for all $f\in\EllonegroupF$ and $g\in\EllpgroupF$, $\rep_f$ is a regular operator on $\EllpgroupF$ for all $f\in\EllpgroupF$, and the map $f\mapsto\rep_f$ is an injective Banach lattice algebra homomorphism $\rep:\EllonegroupF\to\regularops{\EllpgroupF}$.
\end{corollary}

\begin{remark}\label{rem:norm_for_action_of_Ellone_on_Ellp}
In \cref{res:action_of_Ellone_on_Ellp}, in the case where $p=1$, the left regular representation $\rep:\EllonegroupF\to\regularops{\EllonegroupF}$ is an isometric Banach algebra lattice homomorphism. The fact that $\rep$ is an isometry is an immediate consequence of the fact that $\EllonegroupR$ is a Banach algebra with a positive contractive approximate identity. In view of \cref{res:gilbert}, below, we refrain from making a statement on the isometric nature of $\rep$ if $1<p<\infty$.

It is, of course, also possible to prove \cref{res:action_of_Ellone_on_Ellp} directly from the central result \cref{res:general_result_for_semigroups}, by using \cref{res:embedding_of_Lp_into_ordercontcRod} to obtain an embedded copy $X$ of $\EllonegroupR$ and an embedded copy $Y$ of $\EllpgroupR$ in $\contcgroupRod$, and taking $Z$ to be equal to $Y$.
\end{remark}

Now we turn to the case of semigroups, where we shall benefit from the results in Section~\ref{sec:closed_subspaces_of_locally_compact_spaces} via their r{\^o}le in the proof of \cref{res:general_result_for_semigroups}.

\begin{theorem}\label{res:left_regular_representation_of_beurling_algebras}
	Let $\group$ be a locally compact group, and let $S$ be a closed subspace of $\group$ that is a subsemigroup of $\group$. Suppose that $\omega$ is a weight on $S$.
	Then the left regular representation $\rep:\meas(S,\omega,\FF)\to\regularops{\meas(S,\omega,\FF)}$ of the Beurling algebra $\meas(S,\omega,\FF)$ is a Banach lattice algebra homomorphism.
\end{theorem}

\begin{proof}
We again start with the case where $\FF=\RR$. Then \cref{res:embedding_of_beurling_algebras_of_measures_as_banach_lattice_with_properly_separated_elements} shows that $\meas(S,\omega,\RR)$ can be embedded as a vector sublattice of $\od{\contc(S,\RR)}$ and that, after transport of the norm, the embedded copy becomes a Dedekind complete Banach lattice such that its elements with compact, separated support are dense. Completely analogously to the proof of \cref{res:left_regular_representation_of_measures}, \cref{res:general_result_for_semigroups} then shows that the present theorem holds for $\FF=\RR$. As earlier, it then also holds for $\FF=\CC$.
\end{proof}

In a similar vein, we have the following.

\begin{theorem}\label{res:positive_real_line_lattice_homomorphism}
Suppose that $\omega$ is a weight on $\pos{\RR}$. Then the left regular representation $\leftreg:\Ell^1(\pos{\RR},\omega,\FF)\to\regularops{\Ell^1(\pos{\RR},\omega,\FF)}$ of $\Ell^1(\pos{\RR},\omega,\FF)$ is a Banach lattice algebra homomorphism.
\end{theorem}

\begin{proof}
We start with $\FF=\RR$. Since $\omega$ is strictly positive and continuous, \cite[Exercise~7.2.9]{folland_REAL_ANALYSIS_SECOND_EDITION:1999} yields that the measure $\omega\di{x}$ is a regular Borel measure on $\pos{\RR}$. Hence \cref{res:embedding_of_Lp_into_ordercontcRod} shows that $\Ell^1(\pos{\RR},\omega,\RR)$ can be embedded as a vector sublattice of $\od{\contc{(\pos{\RR},\RR)}}$. An application of \cref{res:general_result_for_semigroups} then shows that the present theorem holds for $\FF=\RR$. As earlier, it then also holds for $\FF=\CC$ on general grounds.
\end{proof}

We conclude the results in this section with an application of \cref{res:general_result_for_discrete_semigroups}. The proof will be rather obvious by now, and is left to the reader.

\begin{theorem}
Let $S$ be a cancellative semigroup, and let $\omega$ be a weight on $S$. Then the left regular representation $\rep:\ell^1(S,\omega,\FF)\to\regularops{\ell^1(S,\omega,\FF)}$ of $\ell^1(S,\omega,\FF)$ is an injective Banach lattice algebra homomorphism.
\end{theorem}

We recall that $\ell^1(S,\omega,\FF)$ can be a radical Banach algebra, so that our theorems on left regular representations being Banach lattice algebra homomorphisms are not restricted to the semisimple case.

\begin{remark}\label{rem:left_regular_representation_overview} Let us collect what we know about the left regular representation of a Dedekind complete Banach lattice algebra $\alg$ being a Banach lattice algebra homomorphism from $\alg$ into $\regularops{\alg}$ or not.
	
	On the positive side, we have the following.
	
	\begin{theorem} The left regular representation is a \uppars{real or complex}\ Banach lattice algebra homomorphism from $\alg$ into $\regularops{\alg}$ in the following cases:
		\begin{enumerate}
			\item $\alg$ is the measure algebra $\measgroupF$ of a locally compact group $\group$;
			\item $\alg$ is the group algebra $\EllonegroupF$ of a locally compact group $\group$;
			\item $\alg$ is a Beurling algebra $\meas(S,\omega,\FF)$, where $S$ is a closed subspace of a locally compact group $\group$ that is a subsemigroup of $\group$, and where $\omega$ is a weight on $S$;
			\item $\alg=\Ell^1(\pos{\RR},\omega,\FF)$, where $\omega$ is a weight on $\pos{\RR}$;
			\item $\alg=\ell^1(S,\omega,\FF)$, where $S$ is a cancellative semigroup and $\omega$ is a weight on $S$;
			\item $\alg=\regularops{\vl}$ for a Dedekind complete Banach lattice $\vl$.
		\end{enumerate}
	\end{theorem}

The first five of these results can be found in the present section. The sixth one follows from a result of Synnatzschke's on two-sided multiplication operators; see \cite[Satz~3.1]{synnatzschke:1980}. For further results on two-sided multiplication operators on and between vector lattices of regular operators we refer to \cite{chen_schep:2016,wickstead:2015}.

One could argue that \cref{res:general_result_for_semigroups,res:general_result_for_discrete_semigroups} indicate that the left regular representation will be a Banach lattice algebra homomorphism for very many Dedekind complete Banach lattice algebras on groups or semigroups, whenever the multiplication is akin to a convolution. We are, in fact, not aware of a Dedekind Banach lattice algebra $\alg$ in harmonic analysis where the left regular representation of $\alg$ is not a lattice homomorphism from $\alg$ into $\regularops{\alg}$.
	
	On the negative side, there exist uncountably many two-dimensional, mutually non-isomorphic, commutative Banach lattice algebras with a positive identity element of norm one that have no faithful, finite-dimensional Banach lattice algebra representations at all; see \cite{wickstead:2017a}. In particular, their left regular representations are \emph{not} lattice homomorphisms.
	
	It is unclear if there is an `underlying' property that distinguishes the above Banach lattice algebras on the positive side from those on the negative side. Such a property, and preferably one that is easily verified or falsified in a given case, would be desirable. This question is posed in \cite[Problem~1]{wickstead:2017c}, together with various refinements of it.
	
\end{remark}

\section{Further questions in ordered harmonic analysis}\label{sec:possible_further_research_in_ordered_harmonic_analysis}

\noindent
The previous sections were centred around a convolution-like bilinear map from two Banach lattices on a locally compact (semi)group to a third. There do not seem to be too many results available with the same flavour  of `ordered harmonic analysis', i.e., results that are in the area where harmonic analysis and positivity meet. In this section, we shall discuss results by Arendt, Brainerd and Edwards, and Gilbert that are at this interface and that are related to our results in Section~\ref{sec:lattice_homomorphisms_in_harmonic_analysis}. Our exposition is based on \cite[Section~3]{arendt:1981}, to which the reader is referred for details and additional material. As we shall see, this discussion leads to natural research questions in ordered harmonic analysis. We hope to be able to report on these questions in the future.

\smallskip

Let $\group$ be a locally compact group, and let $1\leq p<\infty$. Then $\group$ acts (not generally isometrically) on $\EllpgroupC$ via the formula
\[
\rho_tf(s)\coloneqq f(st)
\]
for $s,t\in\group$ and $f\in\EllpgroupC$. We shall be interested in operators on $\EllpgroupC$ that commute with all $\rho_s$. To this end, we set
\[
\mathrm{CV}_p(\group,\CC)\coloneqq\desset{T\in\bounded{(\EllpgroupC)}: T\circ\rho_s=\rho_s\circ T\text{ for all }s\in\group}
\]
and
\[
\mathrm{CV}_{p,\mathrm r}(\group,\CC)\coloneqq\desset{T\in\regularops{\EllpgroupC}: T\circ\rho_s=\rho_s\circ T\text{ for all }s\in\group}.
\]
Then $\mathrm{CV}_{p,\mathrm r}(\group,\CC)$ is a complex Banach lattice subalgebra of $\regularops{\EllpgroupC}$. There is an easy proof of this fact, as follows. For $t\in\group$, the map $T\mapsto\rho_{t^{-1}}\circ T\circ\rho_t$ from $\regularops{\EllpgroupC}$ into itself is an algebra automorphism of $\regularops{\EllpgroupC}$. It is a positive map, and since its inverse is clearly also positive\textemdash it is the map $T\mapsto\rho_t\circ T\circ\rho_{t^{-1}}$\textemdash it is a complex Banach algebra lattice automorphism. Hence its fixed point set, which is the commutant of $\rho_t$ in $\regularops{\EllpgroupC}$, is a complex Banach lattice subalgebra of $\regularops{\EllpgroupC}$. Since $\mathrm{CV}_{p,\mathrm r}(\group,\CC)$ is the intersection of these commutants as $t$ ranges over $\group$, the space $\mathrm{CV}_{p,\mathrm r}(\group,\CC)$ is indeed a complex Banach lattice subalgebra of $\regularops{\EllpgroupC}$. This argument is due to Arendt; see \cite[Proof of Proposition~3.3]{arendt:1981}.

There is an easy way to obtain elements of $\mathrm{CV}_{p,\mathrm r}(\group,\CC)$ from elements of $\measgroupC$. Take $\mu\in\measgroupC$, and set (we repeat \cref{eq:mu_convp_g} for convenience)
\begin{equation}\label{eq:left_convolution_with_measure_repeated}
(\mu\convp f)(s) \coloneqq \int_\group\! f(t^{-1}s)\di{\mu(t)}
\end{equation}
for  $f\in\EllpgroupC$ and $s\in\group$. It is easily checked that the (left) convolution operator $\rep_{\mu}$ on $\EllpgroupC$ that is thus defined commutes with all right translations. Obviously, if $\mu$ is positive, then $\rep_{\mu}$ is a positive element of $\mathrm{CV}_{p,\mathrm r}(\group,\CC)$. Conversely, if $T$ is a positive element of $\mathrm{CV}_{p,\mathrm r}(\group,\CC)$, then, according to \cite{brainerd_edwards:1966a}, there exists a positive regular Borel measure $\mu$ on $\group$ such that $Tf$ equals $\mu\convp f$ as in \cref{eq:left_convolution_with_measure_repeated} for all $f\in\contcgroupC$ and $s\in\group$. Note that we do not write that $T=\rep_{\mu}$ because this representation theorem by Brainerd and Edwards does not assert that $\mu$ is a \emph{bounded} measure. When $p=1$ this is always the case, but for $1<p<\infty$ this is related to whether or not $\group$ is amenable. The following result is due to Gilbert; see \cite[Theorem~A]{gilbert:1968} and also \cite[Theorem~18.3.6]{dixmier_C-STAR-ALGEBRAS_ENGLISH_NORTH_HOLLAND_EDITION:1977}, \cite[Theorem~17]{godement:1948}, \cite[Theorem~2.2.1]{greenleaf_INVARIANT_MEANS_ON_TOPOLOGICAL_GROUPS_AND_THEIR_APPLICATIONS:1969}, 
\cite{reiter:1965}, and \cite[Definition~8.3.1, Theorems~8.3.2, and Theorem~8.3.18]{reiter_stegeman_CLASSICAL_HARMONIC_ANALYSIS_AND_LOCALLY_COMPACT_GROUPS_SECOND_EDITION:2000} for the equivalence of various characterisations of amenable locally compact groups.

\begin{theorem}\label{res:gilbert}
	 Let $\group$ be a locally compact group, and let $1<p<\infty$. Then the following are equivalent:
	\begin{enumerate}
		\item $\group$ is amenable;
		\item $\norm{\rep_{\mu}}=\norm{\mu}$ for all $\mu\in\pos{\measgroupC}$;
		\item Whenever $\mu$ is a positive regular Borel measure on $\group$ such that $\mu\convp f$, as defined in \cref{eq:left_convolution_with_measure_repeated}, is in $\EllpgroupC$ for all $f\in\contcgroupC$, and such that there exists a $c\geq 0$ such that $\norm{\mu \convp f}_p\leq c \norm{f}_p$ for all $f\in\contcgroupC$, then $\mu\in\pos{\measgroupC}$.
	\end{enumerate}
\end{theorem}

Suppose that $p=1$ or that $1<p<\infty$ and that $\group$ is amenable. Combining \cref{res:gilbert} with the representation theorem by Brainerd and Edwards, we see that the natural map $\mu\mapsto \rep_{\mu}$ defines a bipositive complex algebra isomorphism $\rep: \measgroupC\to \mathrm{CV}_{p,\mathrm r}(\group,\CC)$ between $\measgroupC$ and $\mathrm{CV}_{p,\mathrm r}(\group,\CC)$. Since we know that $\mathrm{CV}_{p,\mathrm r}(\group,\CC)$ is a complex Banach lattice algebra, this bipositive vector space isomorphism $\rep$ is a complex Banach lattice algebra isomorphism. Since we also know that $\mathrm{CV}_{p,\mathrm r}(\group,\CC)$ is a complex Banach lattice subalgebra of $\regularops{\EllpgroupC}$, we see that the map $\rep:\measgroupC\to \regularops{\EllpgroupC}$ is a complex Banach lattice algebra homomorphism. This fact is a part of the statement of \cite[Proposition~3.3]{arendt:1981}.

\begin{remark}
Allowing ourselves a somewhat imprecise notation, we know from the above that  $\mathrm{CV}_{1,\mathrm{r}}(\group,\CC)=\measgroupC$ and, in addition, that, for $1<p<\infty$, $\mathrm{CV}_{p,\mathrm{r}}(\group,\CC)=\measgroupC$ whenever $\group$ is amenable. Since all bounded operators on an $\Ell^1$-space are regular (see \cite{kantorovich_vulich:1937}), we see that $\mathrm{CV}_{1}(\group,\CC)=\measgroupC$. This is Wendel's theorem; see \cite[Theorem~3.3.40]{dales_BANACH_ALGEBRAS_AND_AUTOMATIC_CONTINUITY:2000}, for example.
\end{remark}

As we know from \cref{res:action_of_measures_on_Ellp}, the map $\rep:\measgroupC\to\regularops{\EllpgroupC}$ is a complex Banach lattice homomorphism for all $p$ such that $1\leq p<\infty$. The amenability of $\group$ is not relevant for this.  As long as one is interested only in $\rep$ being a lattice homomorphism or not, the results in \cite{arendt:1981} are, therefore, not yet optimal. Comparing the machinery needed, including \cite{brainerd_edwards:1966a} and \cite{gilbert:1968}, for the approach in \cite{arendt:1981} on the one hand, with the proof of \cref{res:action_of_measures_on_Ellp} as based on \cref{res:general_result_for_semigroups} on the other hand, one could also argue that\textemdash as long as one is interested only in $\rep$ being a lattice homomorphism or not\textemdash the approach in \cite{arendt:1981} is more complicated than necessary.

Nevertheless, the approach in \cite{arendt:1981} raises a few natural questions, triggered by the description of $\mathrm{CV}_{p,\mathrm r}(\group,\CC)$ that it uses. For example, is there a more general underlying phenomenon that explains what is so special about $p=1$, which is the only case where the amenability of $\group$ does \emph{not} play a r{\^o}le in the description of the regular operators on $\EllpgroupC$ that commute with all right translations? A way to investigate this would be to consider a general Banach function space $\vl$ on $\group$ that is invariant under left and right translations. Under reasonable hypotheses, at least the bounded measures will act on $\vl$ via left convolutions. Is there then a representation theorem as in \cite{brainerd_edwards:1966a} again, stating that a positive operator on $\vl$ that commutes with all right translations is a left convolution with a (possibly unbounded) positive regular Borel measure? What are the properties of $\vl$ that determine whether the amenability of $\group$ is relevant or not for such a measure to be automatically bounded, as in Gilbert's work in \cite{gilbert:1968}?

\smallskip

Let us return to the spaces $\EllpgroupC$. Clearly, $\mathrm{CV}_{p,\mathrm{r}}(\group,\CC)\subseteq\mathrm{CV}_{p}({\group,\CC})$ for all $1\leq p<\infty$. Can the inclusions be proper? For $p=1$, all bounded operators on $\EllonegroupC$ are regular, as was already mentioned above, so in this case equality is automatic. For $1<p<\infty$, we have the following partial answer.

\begin{theorem}\label{res:cowling_fournier_applications}
Let $\group$ be an infinite, amenable, locally compact group, and take $p$ with $1<p<\infty$. Then $\mathrm{CV}_{p,\mathrm{r}}(\group,\CC)\subsetneq\mathrm{CV}_{p}({\group,\CC})$, and so there are bounded operators on $\EllpgroupC$ that commute with all right translations, but are not regular.
\end{theorem}

\begin{proof}
Assume, to the contrary, that $\mathrm{CV}_{p,\mathrm{r}}(\group,\CC)=\mathrm{CV}_{p}({\group,\CC})$ for some $p$ such that $1<p<\infty$. Then $\mathrm{CV}_{p}({\group,\CC})=\measgroupC$. We conclude from this that $\mathrm{CV}_{p}({\group,\CC})=\measgroupC\subseteq \mathrm{CV}_{q}({\group,\CC})$ for all $q$ such that $1<q<\infty$. This, however, contradicts \cite[Theorem~2]{cowling_fournier:1976}.
\end{proof}

This result leads to a few further questions.

First, is the analogue of \cref{res:cowling_fournier_applications} true for more general Banach function spaces $\vl$ on amenable groups that are invariant under left and right translations? To be more specific: for a translation invariant Banach function space $\vl$ on a amenable group, is it true that, whenever there are bounded operators on $\vl$ that are not regular, there are also  bounded operators on $\vl$ that are not regular and that commute with all right translations?

Second, is the amenability of $\group$ a necessary condition in \cref{res:cowling_fournier_applications} for the inclusion to be proper? Put more generally:  for a translation invariant Banach function space $\vl$ on a locally compact group, is it true that, whenever there are bounded operators on $\vl$ that are not regular, there are also  bounded operators on $\vl$ that are not regular and that commute with all right translations?


\subsection*{Acknowledgements} The results in this article were obtained in part when the first author held the Kloosterman Chair in Leiden in October 2017, and when the second author visited Lancaster University in October 2018. The financial support by the Mathematical Institute of Leiden University and the London Mathematical Society is gratefully acknowledged. 




\renewcommand{\btfs}{\mathrm}

\bibliography{general_bibliography}


\end{document}